\documentclass[11pt,reqno]{amsart}
\usepackage{enumerate}
\usepackage{amsrefs}
\usepackage{amsfonts, amsmath, amssymb, amscd, amsthm, bm, cancel,mathrsfs}
\usepackage{url}
\usepackage{graphicx}
\usepackage[
linktocpage=true,colorlinks,citecolor=magenta,linkcolor=blue,urlcolor=magenta]{hyperref}
\usepackage{multicol}
\usepackage{comment}
\usepackage{tikz}
\usepackage[margin=1in]{geometry}

\newtheorem{thm}{Theorem}[section]
\newtheorem*{thmA}{Theorem A}
\newtheorem*{thmA'}{Theorem A'}
\newtheorem*{thmB}{Theorem B}
\newtheorem*{thmC}{Theorem C}

\newtheorem{cor}[thm]{Corollary}

\newtheorem{lem}[thm]{Lemma}
\newtheorem{assump}[thm]{Assumption}
\newtheorem{prop}[thm]{Proposition}

\theoremstyle{definition}

\theoremstyle{remark}
\newtheorem{rem}{Remark}

\numberwithin{equation}{section}
\allowdisplaybreaks

\renewcommand{\widetilde}{\tilde}
\renewcommand{\-}{\overline}

\renewcommand{\b}{\beta}

\renewcommand{\d}{\delta}

\renewcommand{\k}{\kappa}

\renewcommand{\t}{\theta}
\newcommand{\s}{\sigma}

\newcommand{\ra}{\rightarrow}

\begin{document}
	\title[Inverse curvature flows in unit ball]{Inverse curvature flows for capillary hypersurfaces in the unit ball}	
\author[S. Pan]{Shujing Pan}
\address{School of Mathematical Sciences, University of Science and Technology of China, Hefei 230026, P.R. China}
\email{\href{mailto:psj@ustc.edu.cn}{psj@ustc.edu.cn}}
\author[B. Yang]{Bo Yang}
\address{Department of Mathematical Sciences, Tsinghua University, Beijing 100084, P.R. China}
\email{\href{mailto:ybo@tsinghua.edu.cn}{ybo@tsinghua.edu.cn}}
	

	\begin{abstract}
In this paper, we study inverse curvature flows for strictly convex, capillary hypersurfaces in the unit Euclidean ball. We establish the existence and convergence results for a class of such flows. As an application, we derive a family of Alexandrov–Fenchel inequalities for weakly convex hypersurfaces with free boundary.
	\end{abstract}	
	\maketitle
		\tableofcontents
\section{Introduction}

In this paper, we study inverse curvature flows and their applications to the Alexandrov-Fenchel inequalities for capillary hypersurfaces in the unit Euclidean ball. Let $\Sigma\subset\bar{\mathbb{B}}^{n+1} (n\geq 2)$ be a properly embedded smooth hypersurface in the unit Euclidean ball $\bar{\mathbb{B}}^{n+1}$, given by an embedding $x:\bar{\mathbb{B}}^n\to\bar{\mathbb{B}}^{n+1}$ satisfying
\begin{align*}
	\mathrm{int}(\Sigma)=x(\mathbb{B}^n)\subset \mathbb{B}^{n+1},\quad  \partial\Sigma=x(\partial\mathbb{B}^n)\subset \mathbb{S}^n.
\end{align*}
Let $\bar{N}$ be the unit outward normal of $\mathbb S^n=\partial \mathbb B^{n+1}$, and $\nu$ be a smooth choice of the unit normal of $\Sigma$. For $\theta\in(0,\pi)$, we say that $\Sigma$ has  $\theta$-capillary boundary $\partial\Sigma$ on $\mathbb{S}^n$ if $\Sigma$ intersects $\mathbb S^n$ at the constant contact angle $\t$, that is,
\begin{equation*}
	\langle\bar{N}\circ x,\nu\rangle =-\cos\theta, \quad \text{along $\partial \mathbb B^n$}.
\end{equation*}
In particular, if $\theta=\frac{\pi}{2}$, i.e., $\Sigma$ intersects $\mathbb S^n$ orthogonally, we call that $\Sigma$ has free boundary. Denote by $\widehat{\partial\Sigma}$ the convex body in $\mathbb{S}^n$ enclosed by $\partial \Sigma$, and $\widehat{\Sigma}$ the convex domain in $\bar{\mathbb{B}}^{n+1}$ enclosed by $\widehat{\partial\Sigma}$ and $\Sigma$. We choose the unit normal $\nu$ of $\Sigma$ as the one pointing outward of $\widehat{\Sigma}$.

Two model examples are the spherical cap of radius $r$ around a constant unit vector $e\in \mathbb{S}^n\subset\mathbb R^{n+1}$ with $\t$-capillary boundary, given by
\begin{align}\label{s1:spherical-cap}
	C_{\theta,r}(e):=\left\{x\in \bar{\mathbb B}^{n+1}:\left| x-\sqrt{r^2+2r\cos\t+1}e\right|=r\right\},
\end{align}
and the {flat ball} around $e$ with $\t$-capillary boundary, given by
\begin{align}\label{s1:flat-disk}
	C_{\theta,\infty}(e):=\{ x\in \bar{\mathbb B}^{n+1}:\langle x,e\rangle=\cos\t\}.
\end{align}

Let $\kappa=(\kappa_1,\cdots,\kappa_n)$ be the principal curvatures of $\Sigma$, we define the $k$th normalized mean curvature of $\Sigma$ as 
\begin{equation*}
	E_k(\kappa)=\binom{n}{k}^{-1}\sum_{1\leq i_1<i_2<\cdots<i_k\leq n}{\kappa_{i_1}\kappa_{i_2}\cdots\kappa_{i_n}}.
\end{equation*}
In particular, we write $H=nE_1$ for the mean curvature of $\Sigma$, and $K=E_n$ as the Gauss curvature of $\Sigma$. We say a hypersurface $\Sigma$ is {\em strictly convex}, if all principal curvatures $\k_i>0$ everywhere on $\Sigma$, and is {\em weakly convex} if all principal curvatures $\kappa_i\geq 0$ everywhere on $\Sigma$.

This paper deals with the expanding curvature flows of the form
\begin{equation}\label{Inverse-flow}
	\left\{\begin{aligned}
		\left(\partial_t x\right)^{\bot}&=\frac{1}{F}\nu &\text{in}\quad \bar{\mathbb{B}}^n \times[0,T),\\
		\langle\bar{N}\circ x,\nu\rangle&=-\cos\theta  &\text{on}\quad \partial\bar{\mathbb{B}}^n \times[0,T),\\
		x(\cdot,0)&=x_0(\cdot)  &\text{on} \quad \bar{\mathbb{B}}^n,
	\end{aligned}\right.
\end{equation}
where  $F=F(\mathcal{W})$  is a smooth function evaluated at the Weingarten operator $\mathcal{W}=\{h_i^j\}$ of of the flow
hypersurfaces  $\Sigma_t=x(\bar{\mathbb{B}}^n,t)$ and $\nu$ is the outward pointing
normal. 

Without ambiguity, we also  regard $F$ as a function of the eigenvalues $\kappa(\mathcal{W})$ of the Weingarten operator $\mathcal{W}$. In this context, we denote 
$$F_i(\kappa)=\frac{\partial F}{\partial\kappa_i}(\kappa).$$
Moreover, the following relations hold:
\begin{equation*}
	F_j^i=\frac{\partial F}{\partial h_i^j}=g_{kj}F^{ki}=g_{kj}\frac{\partial F}{\partial h_{ki}}.
\end{equation*}
There are several conditions that are consistently imposed on the curvature function $F$ throughout this paper:
\begin{assump}\label{assum}
	Let $\Gamma\subset\mathbb{R}^n$ be a symmetric, open and convex cone containing the positive cone 
	\begin{equation}
		\Gamma^{+}=\{\kappa\in\mathbb{R}^n|\kappa_i>0,\forall 1\leq i\leq n\}. 
	\end{equation}
	Suppose $F\in C^0(\bar{\Gamma})\cap C^2(\Gamma)$  satisfies 
	\begin{itemize}
		\item[(i)] $F$ is strictly positive in $\Gamma$ and vanishes on the boundary of $\Gamma$ (i.e. $F|_{\Gamma}>0$ and $F|_{\partial\Gamma}=0$), which is normalized such that $F(1,\cdots,1)=1$;
		\item[(ii)] $F$ is homogeneous of degree one: $F(\lambda\kappa)=\lambda F(\kappa)$, for any $\lambda>0$;
		\item[(iii)] $F$ is  strictly increasing in each argument: $F_i(\kappa)=\frac{\partial F}{\partial\kappa_i}(\kappa)>0$, $\forall i=1,\cdots,n$;
		\item[(iv)] $F$ is concave and inverse concave on $\Gamma^+$. That is, both $F$ and the function $F_{*}$ given by
		\begin{align*}
			F_{*}(\kappa_1,\cdots,\kappa_n)=F(\kappa_1^{-1},\cdots,\kappa_n^{-1})^{-1}
		\end{align*}
		are concave on $\Gamma^+$.
	\end{itemize}
\end{assump}
\begin{rem}
	The most important examples of  curvature functions $F$ satisfying the above assumption are $\left(\frac{E_k}{E_{\ell}}\right)^{\frac{1}{k-\ell}}$, $0\leq \ell<k\leq n$.
\end{rem}

We prove the following  existence and convergence result:
\begin{thm}
Suppose that $F$ is a curvature function satisfying Assumption \ref{assum}. Let $\Sigma\subset\bar{\mathbb{B}}^{n+1} (n\geq 2)$ be a properly embedded, strictly convex smooth hypersurface with capillary boundary supported on $\mathbb{S}^n$ at a contact angle $\theta\in(0,\frac{\pi}{2}]$, given by an embedding $x:\bar{\mathbb{B}}^n\to\bar{\mathbb{B}}^{n+1}$. Then there exists a finite time $T^{*}<\infty,\alpha>0$ such that the flow \eqref{Inverse-flow} starting from $\Sigma$ admits a unique solution $x(\cdot,t)$, such that
	\begin{equation*}
		x(\cdot,t)\in C^{\infty}(\bar{\mathbb{B}}^n\times(0,T^{*}))\cap C^{2+\alpha,1+\frac{\alpha}{2}}(\bar{\mathbb{B}}^n\times[0,T^{*})).
	\end{equation*}
	Moreover, the solution $x(\cdot,t)$ remains strictly convex in the time interval $(0,T^{*})$ and converges to a flat ball with $\theta$-capillary boundary as $t\to T^{*}$.
\end{thm}

The study of inverse curvature flows for closed hypersurfaces in Euclidean space was pioneered by Gerhardt \cite{Ger90} and Urbas \cite{Urbas90}, who established the long-time existence and convergence of the flow for star-shaped initial  hypersurfaces with positive curvature function  $F$. This result, particularly with $F=\frac{E_{k+1}}{E_{k}}$ was later exploited by Guan-Li \cite{guanli09} to generalise the Alexandrov–Fenchel quermassintegral inequalities from the convex setting to the star-shaped and $(k+1)$-convex setting. Since then, a series of similar results have been developed using the same method in various ambient spaces, and the study of inverse curvature flows has become increasingly active.  For examples, inverse curvature flows in hyperbolic space and the sphere have been investigated in \cite{Ger11,Ger15,Liu17,Yu}, while analogous results in more general warped product spaces can be found in  \cite{BHW16,LW17,Scheuer17,Scheuer19}, among many others. A more comprehensive introduction to this topic is available in \cite{Scheuer19}.

For the free boundary case, Lambert and Scheuer \cite{Lambert-Scheuer2016} has shown the convergence result of the flow \eqref{Inverse-flow} with $F=\frac{1}{H}$  (i.e. the inverse mean curvature flow) for strictly convex hypersurfaces in $\bar{\mathbb{B}}^{n+1}$. As an application, they obtained the following Li-Yau type inequality for convex hypersurfaces with free boundary in any dimension $n\geq 3$.

\begin{thmA}[\cite{Lambert-Scheuer2017}]
	Let $\Sigma\subset\bar{\mathbb{B}}^{n+1} (n\geq 3)$ be a properly embedded, strictly convex smooth hypersurface with free boundary supported on $\mathbb{S}^n$ . Then there holds
	\begin{equation}
		\frac{1}{2}|\Sigma|^{\frac{2-n}{n}}\int_{\Sigma}{H^2}\,d\mu+b_n^{\frac{2-n}{n}}|\partial\Sigma|\geq b_n^{\frac{2-n}{n}}|\mathbb{S}^{n-1}|,
	\end{equation}
where $b_n$ denotes the volume of $n$-dimensional unit ball, and equality holds if and only if $\Sigma$ is a flat disk with free boundary.
\end{thmA}

Scheuer, Wang and Xia \cite{Scheuer-Wang-Xia2018} introduced the quermassintegrals $W_k(\widehat{\Sigma})$ for smooth hypersurfaces $\Sigma$ with free boundary in $\mathbb{B}^{n+1}$. Later,  Weng and Xia \cite{Weng-Xia2022} generalized this definitions to $W_{k,\theta}(\widehat{\Sigma})$ for capillary hypersurfaces with contact angle $\theta\in(0,\pi)$ which exhibit favorable variational properties (see \S\ref{sub-sec2.2} for details). Using a locally constrained inverse curvature flow, they established  the following Alexandrov-Fenchel inequalities for weakly convex hypersurfaces in $\bar{\mathbb{B}}^{n+1}$.

\begin{thmB}[\cites{Scheuer-Wang-Xia2018, Weng-Xia2022}]
	Let $\Sigma\subset\bar{\mathbb{B}}^{n+1} (n\geq 2)$ be a properly embedded, weakly convex smooth hypersurface in the unit ball with $\theta$-capillary boundary, where $\t\in (0,\frac{\pi}{2}]$. Then for any $k=0,\cdots, n-1$, there holds
	\begin{equation}\label{eq-AF2}
		W_{n,\theta}(\widehat{\Sigma})\geq (f_n\circ f_k^{-1})(W_{k,\theta}(\widehat{\Sigma})),
	\end{equation}
	where $f_k(r)$ is the strictly increasing function given by $f_k(r)=W_{k,\theta}(\widehat{C_{\theta,r}})$. Equality holds in \eqref{eq-AF2} if and only if $\Sigma$ is a spherical cap or a flat ball with $\theta$-capillary boundary.
\end{thmB}

Later, the second author with Hu, Wei and Zhou \cite{HWYZ-2023} proved the following Alexandrov-Fenchel inequalities for weakly convex hypersurfaces in $\bar{\mathbb{B}}^{n+1}$.

\begin{thmC}[\cite{HWYZ-2023}]
	Let $\Sigma\subset\bar{\mathbb{B}}^{n+1} (n\geq 2)$ be a properly embedded, weakly convex smooth hypersurface in the unit ball with $\theta$-capillary boundary, where $\t\in (0,\frac{\pi}{2}]$. Then for any $k=1,\cdots, n-1$,  there holds
	\begin{equation}\label{eq-AF}
		W_{k,\theta}(\widehat{\Sigma})\geq (f_k\circ f_0^{-1})(W_{0,\theta}(\widehat{\Sigma})),
	\end{equation}
	where $f_k(r)$ is the strictly increasing function given by $f_k(r)=W_{k,\theta}(\widehat{C_{\theta,r}})$. Equality holds in \eqref{eq-AF} if and only if $\Sigma$ is a spherical cap or a flat ball with $\theta$-capillary boundary.
\end{thmC}

In the second part of this paper, we prove a family of Alexandrov-Fenchel inequalities for weakly convex hypersurfaces in $\bar{\mathbb{B}}^{n+1}$ with free boundary (i.e. $\theta=\frac{\pi}{2}$) using the inverse mean curvature flow. In this case, we denote the quermassintegral $W_{k,\frac{\pi}{2}}(\widehat{\Sigma})$ as $W_k(\widehat{\Sigma})$ for the sake of simplicity.

\begin{thm}\label{Thm-A F}
	Let $\Sigma\subset\bar{\mathbb{B}}^{n+1} (n\geq 2)$ be a properly embedded, weakly convex smooth hypersurface in the unit ball with free boundary. Then for $k\in\mathbb{N}_{+}$ and $2k+1\leq n$, there holds
	\begin{equation}\label{In-A-F}
		W_{2k+1}(\widehat{\Sigma})\geq\frac{\omega_{n-1}}{n}\frac{\prod_{j=0}^k(n-2j)}{\prod_{j=0}^k(n+1-2j)}\sum_{i=0}^k(-1)^i\binom{k}{i}\frac{1}{n-2k+2i}\left[\frac{n(n+1)}{\omega_{n-1}}W_1(\widehat{\Sigma})\right]^{\frac{n-2k+2i}{n}},
	\end{equation}
where $\omega_{n-1}$ is the area of $(n-1)$-dimensional unit sphere, and the equality \eqref{In-A-F} holds if and only if $\Sigma$ is a flat disk with free boundary.
\end{thm}
The rest of the paper is organized as follows: In \S\ref{Sec-pre}, we collect some preliminaries which will be used in this paper, including the geometry of hypersurfaces in the ball with capillary boundary, the definition and variational formulas of quermassinegrals for hypersurfaces with $\theta$-capillary boundary and some basic properties for curvature functions and evolution equations along the inverse curvature flow \eqref{Inverse-flow}. In \S\ref{Sec-cur}, we give the $C^2$ estimate along the inverse curvature flow \eqref{Inverse-flow}. In particular, we show that the inverse curvature flow \eqref{Inverse-flow} starting from a strictly convex hypersurface exists for a finite time and remains strictly convex. In \S\ref{Sec-Con}, we show that the flow \eqref{Inverse-flow} converges to a flat ball with $\theta$-capillary boundary around some $e\in\mathbb{S}^n$. In \S\ref{Sec-A F}, we prove the Alexandrov Fenchel inequalities \eqref{In-A-F} for weakly convex hypersurfaces with free boundary.

\textbf{Acknowledge.} The authors would like to express their sincere gratitude to  Prof. Julian Scheuer  for his generous and insightful discussions, particularly his crucial advice on the uniform ellipticity of the curvature function $F$.  The research was supported by National Key R and D Program of China 2021YFA1001800 and 2020YFA0713100, China Postdoctoral Science Foundation No.2024M751605, and Shuimu Tsinghua Scholar Program (No. 2023SM102).
 
\section{Preliminaries}\label{Sec-pre}
\subsection{Hypersurfaces in the ball with capillary boundary} 
Let $\theta\in (0,\frac{\pi}{2}]$. Suppose that $\Sigma\subset\bar{\mathbb{B}}^{n+1}$ is a smooth, properly embedded, weakly convex hypersurface with $\theta$-capillary boundary, which is given by an embedding $x:\bar{\mathbb B}^n\ra \Sigma$ such that
$$
\mathrm{int}(\Sigma)=x({\mathbb B}^{n}) \subset \mathbb B^{n+1}, \quad \partial \Sigma=x(\partial \mathbb B^{n})\subset \mathbb S^n.
$$
Then $\partial\Sigma\subset\mathbb{S}^{n}$ is a weakly convex hypersurface of the unit sphere and bounds a weakly convex body in $\mathbb{S}^n$, which we denote by $\widehat{\partial\Sigma}$, cf. \cite[Theorem 1.1]{CW1970}. Let $\widehat{\Sigma}$ denote  the bounded domain in $\bar{\mathbb{B}}^{n+1}$ enclosed by $\Sigma$ and $\widehat{\partial\Sigma}$. We denote by $\nu$ the unit normal field of $\Sigma$, and $\bar{N}$ the position vector of $\mathbb{S}^n$. The outward pointing conormal vector of $\partial\Sigma\subset\Sigma$ is denote by $\mu$, and the unit normal of $\partial \Sigma$ in $\mathbb{S}^n$ is denote by $\bar{\nu}$, chosen such that the pairs $\{\nu,\mu\}$ and $\{\bar{\nu},\bar{N}\circ x\}$ induce the same orientation in the normal bundle of $\partial\Sigma\subset\mathbb{S}^n$.

It follows from $\langle \bar{N}\circ x,\nu\rangle=-\cos \t$ that
\begin{equation}\label{normal transform}
	\left\{\begin{aligned}
		\bar{N}\circ x=&\sin\theta\mu-\cos\theta\nu,\\
		\bar{\nu}=&\cos\theta\mu+\sin\theta\nu.
	\end{aligned}\right.
\end{equation}
Or equivalently,
\begin{equation}\label{normal transform-II}
\left\{\begin{aligned}
	\mu=&\sin\t \bar{N} \circ x+\cos\t \bar{\nu},\\
	\nu=&-\cos\t \bar{N}\circ x+\sin\t \bar{\nu}.
	\end{aligned}\right.
\end{equation}

We denote by $D$ the Levi-Civita connection of $\mathbb R^{n+1}$ with respect to the Euclidean metric $\delta$, and $\nabla$ the Levi-Civita connection on $\Sigma$ with respect to the induced metric from the embedding $x:\bar{\mathbb B}^n \to \Sigma \subset \mathbb R^{n+1}$. The second fundamental form of $\Sigma$ in $\mathbb R^{n+1}$ is given by
\begin{equation*}
	h(X,Y):=-\langle D_X Y,\nu\rangle,\ \ X,Y\in T(\Sigma).
\end{equation*}

Note that $\partial\Sigma$ can be regarded as a smooth closed hypersurface both in $\mathbb{S}^n$ and $\Sigma$. The second fundamental form of $\partial\Sigma$ in $\mathbb{S}^n$ is given by
\begin{equation*}
	\widehat{h}(X,Y):=-\langle\nabla^{\mathbb{S}^n}_X Y,\bar{\nu}\rangle=-\langle D_XY,\bar{\nu}\rangle,\ \ X,Y\in T(\partial\Sigma).
\end{equation*}
The second fundamental form of $\partial\Sigma$ in $\Sigma$ is given by
\begin{equation*}
	\tilde{h}(X,Y):=-\langle\nabla_X Y,\mu\rangle=-\langle D_XY,\mu\rangle,\ \ X,Y\in T(\partial\Sigma).
\end{equation*}

\begin{prop}[\cite{Weng-Xia2022}]\label{boundary-h-property}
	Let $\Sigma\subset\bar{\mathbb{B}}^{n+1}$ be a $\theta$-capillary hypersurface. Let $\{e_\alpha\}_{\alpha=1}^{n-1}$ be an orthonormal frame of $\partial\Sigma$. Then the following relations hold on $\partial \Sigma$:\\
	$(1)$ $\mu$ is a principal direction of $\Sigma$, i.e., $h(\mu,e_\alpha)=0$.\\
	$(2)$ $h_{\alpha\beta}=\sin\theta\widehat{h}_{\alpha\beta}-\cos\theta\delta_{\alpha\beta}$.\\
	$(3)$ $\tilde{h}_{\alpha\beta}=\cos\theta\widehat{h}_{\alpha\beta}+\sin\theta\delta_{\alpha\beta}=\cot\theta h_{\alpha\beta}+\frac{1}{\sin\theta}\delta_{\alpha\beta}$.\\
	$(4)$ $\nabla_\mu h_{\alpha\beta}=\tilde{h}_{\beta\gamma}\left(h_{\mu\mu}\delta_{\alpha\gamma}-h_{\alpha\gamma}\right)$.
\end{prop}
In the following of this subsection, we derive an estimate of the height function $\langle x,e\rangle$. Weng and Xia (\cite[Proposition 2.16]{Weng-Xia2022}, see also \cite[Proposition 2.5]{HWYZ-2023}) have proved that for a strictly convex hypersurface $\Sigma$, there exists a unit vector $e\in\mathbb{S}^n$ such that 
\begin{equation}\label{es-h-1}
	\langle x,e\rangle\geq \cos\theta+\delta_0,
\end{equation}
where $\delta_0>0$ is a positive constant depending only on $\Sigma$. Here, we present a new proof of this estimate,  which explains the concrete geometric quantities on which  $\delta$ depends. The underlying idea is similar to that in \cite{Lambert-Scheuer2016}. We begin with a technical lemma.
\begin{lem}\label{Lem-2.2}
	Let $R>0$, $e_0\in\mathbb{R}^{n+1}$ be a unit vector and $C\subset\mathbb{R}^{n+1}$ be a convex closed cone whose vertex at $\cos\theta e_0$. Then for all $\varepsilon>0$, there exists $\delta>0$, such that
	\begin{equation*}
		\langle a-\cos\theta e_0,e_0\rangle\geq \cos\left(\frac{\pi}{2}-\varepsilon\right)||a-\cos\theta e_0||,\quad \forall a\in C
	\end{equation*}
implies
\begin{equation*}
	\langle x-\cos\theta e_0,e_0\rangle\geq R+\delta,\quad \forall B_R(x)\subset C.
\end{equation*}	
\end{lem}
\begin{proof}
	The proof is similarly as in \cite[Lemma 10]{Lambert-Scheuer2016}, except that now we translate the vertex of the convex cone from origin $O$ to $\cos\theta e_0$.
\end{proof}
\begin{lem}\label{Lem-es-height}
	Let $\Sigma$ be a strictly convex hypersurface in $\bar{\mathbb{B}}^{n+1}$ with $\theta$-capillary boundary with $\theta\in(0,\frac{\pi}{2}]$. Let $e\in\mathrm{int}(\widehat{\partial\Sigma})$ be a direction, such that $\widehat{\partial\Sigma}\subset\mathrm{int}(\widehat{\partial C_{\theta,\infty}}(e))$. Then there exists a positive constant $\delta>0$, depending only on the length of the second fundamental form of $\Sigma$, the distance from $e$ to $\partial\Sigma$ and the distance from $\partial\Sigma$ to $\partial C_{\theta,\infty}(e)$ such that
	\begin{equation}
		\langle x,e\rangle\geq \cos\theta+\delta.
	\end{equation}
\end{lem}
\begin{proof}
	The case that $\theta=\frac{\pi}{2}$ was already established by Lambert and Scheuer in \cite[Lemma 12]{Lambert-Scheuer2016}. Therefore, we only  consider the case $\theta\in(0,\frac{\pi}{2})$. The existence of such $e$ can be found in \cite[Proposition 2.16]{Weng-Xia2022}. We denote the cone whose vertex at $\cos\theta e$ as 
	\begin{equation}\label{defn-cone}
		C_0=\{y\in\mathbb{R}^{n+1}:y=s(x-\cos\theta e)+\cos\theta e, s\geq 0,x\in\widehat{\partial\Sigma}\}.
	\end{equation}
We claim that $\Sigma\subset C_0\cap \bar{\mathbb{B}}^{n+1}$.  By \cite[Proposition 2.15]{Weng-Xia2022}, for any $p\in\partial\Sigma$, $\Sigma$ lies on one side of $T_{p}{\Sigma}$. More precisely, we have
\begin{equation}\label{half-1}
	\langle x-p,\nu(p)\rangle\leq 0, \forall x\in\Sigma.
\end{equation}
On the other hand,
\begin{align}
	&\langle s(p-\cos\theta e)+\cos\theta e-p,\nu(p)\rangle\notag\\
	=&(1-s)\cos\theta(1+\langle \nu(p),e\rangle)\geq 0,\label{half-2}
\end{align}
and the equality in \eqref{half-2} holds if and only if $s=1$. Combining \eqref{half-1} and \eqref{half-2} yields the claim. 

Assume that the function $\langle x-\cos\theta e,e\rangle$ attains its minimum  on $\Sigma$ at a point $a$ and 
\begin{equation}\label{distepsilon}
	\mathrm{dist}(\partial\Sigma,C_{\theta,\infty}(e))=\varepsilon.
\end{equation}
Then two cases arise:

(i) if $\mathrm{dist}(a,\widehat{\partial\Sigma})\leq\frac{\varepsilon}{2}$,  it follows directly that $\langle a,e\rangle\geq \cos\theta+\frac{\varepsilon}{2}$.

(ii) if $\mathrm{dist}(a,\widehat{\partial\Sigma})>\frac{\varepsilon}{2}$, then by the same argument in \cite[Lemma 12]{Lambert-Scheuer2016}, there exists $R>0$ depends on the the length of the second fundamental form of $\Sigma$ and $\langle\nu,e\rangle$, such that $B_R$ is an interior sphere at $a$. By \cite[Proposition 2.16]{Weng-Xia2022}, $\langle\nu,e\rangle<-\cos^2{\theta}$ and hence $R$ depends on the the length of the second fundamental form of $\Sigma$ and $\theta$. If we take $r=\min\{R,\frac{\varepsilon}{8}\}$, then $B_r$ is also an interior ball at $a$ with radius $r$. Moreover, for $\forall x\in B_r$, we have
\begin{equation}
	\mathrm{dist}(x,\widehat{\partial\Sigma})\geq \mathrm{dist}(a,\widehat{\partial\Sigma})-2r\geq \frac{\varepsilon}{4},
\end{equation}
which implies that $B_r\subset C_0\cap\bar{\mathbb{B}}^{n+1}$. By lemma \ref{Lem-2.2}, there exists a $\delta>0$, such that $\langle a,e\rangle\geq \cos\theta+\delta$.
\end{proof}

\subsection{Quermassintegrals for hypersurfaces with $\theta$-capillary boundary}\label{sub-sec2.2}
The quermassintegrals of weakly convex hypersurfaces were first introduced  by Scheuer, Wang and Xia \cite{Scheuer-Wang-Xia2018} for hypersurfaces with free boundary in $\mathbb{B}^{n+1}$. Later, Weng and Xia \cite{Weng-Xia2022} generalized the definition to  hypersurfaces with $\theta$-capillary boundary  in $\mathbb{B}^{n+1}$. We briefly recall the definition here. 

Let $\Sigma$ be a weakly convex hypersurface in $\mathbb{B}^{n+1}$ with $\theta$-capillary boundary. Then $\partial\Sigma\subset\mathbb{S}^{n}$ is a weakly convex hypersurface of the unit sphere, which bounds a weakly convex body in $\mathbb{S}^n$, denote by $\widehat{\partial\Sigma}$. Denote by $\widehat{\Sigma}$ the bounded domain in $\bar{\mathbb{B}}^{n+1}$ enclosed by $\Sigma$ and $\widehat{\partial\Sigma}$. The quermassintegrals for $\widehat{\Sigma}$ are defined as: 
\begin{align}
	W_{0,\theta}(\widehat{\Sigma})=&|\widehat{\Sigma}|, \label{s2:quermassintegral-0}\\
	W_{1,\theta}(\widehat{\Sigma})=&\frac{1}{n+1}\left(|\Sigma|-\cos\theta W_0^{\mathbb{S}^n}(\widehat{\partial\Sigma})\right),
\end{align}
and for $1\leq k\leq n-1$,
\begin{align}\label{s2:quermassintegral-2}
	\begin{aligned}
		&W_{k+1,\theta}(\widehat{\Sigma})=\frac{1}{n+1}\left\{\int_{\Sigma}E_{k}dA-\cos\theta\sin^k\theta W_k^{\mathbb{S}^n}(\widehat{\partial\Sigma})\right.  \\
		&\quad \quad \quad \left.-\sum_{\ell=0}^{k-1}\frac{(-1)^{k+\ell}}{n-\ell}\binom{k}{\ell}\left[(n-k)\cos^2\theta+k-\ell\right]\cos^{k-1-\ell}\theta\tan^\ell\theta W_\ell^{\mathbb{S}^n}(\widehat{\partial\Sigma})\right\},
	\end{aligned}
\end{align}
where $W_{k}^{\mathbb{S}^n}(\widehat{\partial \Sigma}), 0\leq k\leq n$ are the quermassintegrals in $\mathbb{S}^n$ which are defined by (see \cite{Sol06})
\begin{align}
	W_0^{\mathbb S^n}(\widehat{\partial \Sigma})=&|\widehat{\partial \Sigma}|, \label{W_0}\\
	W_1^{\mathbb S^n}(\widehat{\partial \Sigma})=&\frac{1}{n}|\partial \Sigma|,\label{W_1}\\
	W_{k}^{\mathbb S^n}(\widehat{\partial \Sigma})=&\frac{1}{n}\int_{\partial \Sigma} E_{k-1}^{\mathbb S^n} ds+\frac{k-1}{n-k+2}W^{\mathbb S^n}_{k-2}(\widehat{\partial \Sigma}), \quad 2\leq k\leq n,\label{W_k}
\end{align}
where $ds$ is the area element and $E_{k}^{\mathbb S^n}$ is the $k$th normalized mean curvature of the hypersurface $\partial \Sigma$ in $\mathbb S^n$.

The quermassintegrals defined above satisfy the following variational property.

\begin{prop}[\cite{Weng-Xia2022}]
	Let $\Sigma_t\subset\bar{\mathbb{B}}^{n+1}$ be a family of smooth, properly embedded hypersurfaces with $\theta$-capillary boundary on $\mathbb{S}^n$, given by $x(\cdot,t):\bar{\mathbb{B}}^n\rightarrow \bar{\mathbb{B}}^{n+1}$, such that
	\begin{equation*}
		\left(\partial_tx\right)^\perp=f\nu
	\end{equation*}
	for some normal speed function $f$. Then there holds
	\begin{equation}\label{s2.dW}
		\frac{d}{dt}W_{k,\theta}(\widehat{\Sigma}_t)=\frac{n+1-k}{n+1}\int_{\Sigma_t}E_kfd\mu_t, \quad 0\leq k\leq n,
	\end{equation}	
		and
		\begin{equation}
		\frac{d}{dt}W_{n+1,\theta}(\widehat{\Sigma_t})=0.
			\end{equation}
\end{prop}
\subsection{Basic properties for curvature function} In this subsection,  we collect some basic properties for curvature functions under mild assumptions for later use.

	\begin{lem}[\cite{guan2014}]\label{Lemma-NMI}
		Let $1\leq k\leq n-1$. For $\kappa\in\Gamma_k^{+}$, there holds
		\begin{equation}\label{In-NewMac}
			E_{k+1}(\kappa)E_{k-1}(\kappa)\leq E_k^2(\kappa),\quad E_{k+1}(\kappa)\leq E_k^{\frac{k+1}{k}}(\kappa).
		\end{equation}
		Equality holds in \eqref{In-NewMac} if and only if $\kappa_1=\cdots=\kappa_n$.
	\end{lem}

The following property can be found in \cite[Lemma 2.2.19, Lemma 2.2.20]{Ger06}.
\begin{lem}[\cite{Ger06}]
	Let $F\in C^2(\Gamma^{+})\cap C^0(\bar{\Gamma}^{+})$ be a strictly monotone, concave curvature function, positively homogeneous of degree 1 with $F(1,\cdots,1)>0$, then
	\begin{equation}\label{re-F-H}
		F(\kappa)\leq F(1,\cdots,1)E_1(\kappa),
	\end{equation}
and
\begin{equation}\label{sum-F}
	\sum_{i=1}^n F_i(\kappa)\geq F(1,\cdots,1),
\end{equation}
where $\kappa=(\kappa_i)_{i=1}^n\in\Gamma^{+}$.
\end{lem}

\begin{lem}[\cites{And1994,Ecker-Huisken89}]\label{lem-concave}
	Suppose that $F$ is a concave function and $\kappa\in\Gamma^{+}$, then there holds
	\begin{equation}
		\left(\frac{\partial F}{\partial\kappa_i}-\frac{\partial F}{\partial\kappa_j}\right)(\kappa_i-\kappa_j)\leq 0.
	\end{equation}
\end{lem}

\begin{lem}[\cite{Ger96}]\label{lem-ger96}
	Let $F$ be a curvature function defined on $\Gamma^{+}$, which is concave and inverse concave. Let $b^{ij}$ be the inverse matrix of $h_{ij}$, then
	\begin{equation}\label{pro-con-in}
		\ddot{F}^{ij,k\ell}\eta_{ij}\eta_{k\ell}+2\dot{F}^{ik}b^{j\ell}\eta_{ij}\eta_{k\ell}\geq 2F^{-1}(\dot{F}^{ij}\eta_{ij})^2,
	\end{equation}
for any real symmetric $n\times n$ matrix $(\eta_{ij})$. 
\end{lem}

\subsection{Evolution equations} In this subsection, we derive some evolution equations along the flow \eqref{Inverse-flow} for later use. Along a general flow
\begin{equation}\label{general-flow}
	\left\{\begin{aligned}
		\left(\partial_t x\right)^{\bot}&=f\nu &\text{in}\quad \bar{\mathbb{B}}^n \times[0,T),\\
		\langle\bar{N}\circ x,\nu\rangle&=-\cos\theta  &\text{on}\quad \partial\bar{\mathbb{B}}^n \times[0,T),\\
		x(\cdot,0)&=x_0(\cdot)  &\text{on} \quad \bar{\mathbb{B}}^n,
	\end{aligned}\right.
\end{equation}
where $f$ is the normal speed function. As in \cite[\S 2.4]{Weng-Xia2022}, the $\theta$-capillary boundary condition along the flow requires that the tangential component $\mathcal{T}=(\partial_tx)^\top\in T\Sigma_t$ must satisfy $\mathcal{T}\big|_{\partial\Sigma_t}=f\cot\theta\mu+\tilde{\mathcal{T}}$,  where $\tilde{T}\in T\left(\partial\Sigma_t\right)$. Therefore, up to a time-dependent diffeomorphism of $\partial\mathbb{B}^n$, we can assume that $\tilde{T}=0$ and hence
\begin{equation}\label{con-T}
	\mathcal{T}\big|_{\partial\Sigma_t}=f\cot\theta\mu.
\end{equation}
The $\theta$-capillary boundary condition also ensures that (see \cite[Prop. 2.12]{Weng-Xia2022})
\begin{equation}\label{de-boundary}
	\nabla_\mu f=\left(\frac{1}{\sin\theta}+\cot\theta h_{\mu\mu}\right)f\ \ \text{along}\ \partial\Sigma_t.
\end{equation}
Thus,  the flow \eqref{general-flow} can be equivalently rewritten in the following form:
\begin{equation}\label{s3:GL-flow}
	\left\{\begin{aligned}
		\partial_t x&=f\nu+\mathcal{T} \qquad &\text{in $\bar{\mathbb{B}}^n \times[0,T)$},\\
		\langle\bar{N}\circ x,\nu\rangle&=-\cos\theta &\text{on  $\partial\bar{\mathbb{B}}^n \times[0,T)$},\\
		x(\cdot,0)&=x_0(\cdot) &\text{on $\bar{\mathbb{B}}^n$}.
	\end{aligned}\right.
\end{equation}

\begin{prop}[\cite{Weng-Xia2022}]\label{Prop-general-evo}
	Along the general flow \eqref{s3:GL-flow}, it holds that
	\begin{align}
		&\partial_t{g_{ij}}=2f h_{ij}+\nabla_i \mathcal{T}_j+\nabla_j \mathcal{T}_i,\label{ev-metric}\\
		&\partial_t d\mu_t=(fH+\mathrm{div}(\mathcal{T}))d\mu_t,\label{ev-area}\\
		&\partial_t\nu=-\nabla f+g^{ij}h(e_i,\mathcal{T})e_j,\label{ev-normal}\\
		&\partial_t h_{ij}=-\nabla_i\nabla_j f+f(h^2)_{ij}+\nabla_\mathcal{T}{h_{ij}}+h_j^k\nabla_i{\mathcal{T}_k}+h_i^k\nabla_j{\mathcal{T}_k},\label{ev-sfn}\\
		&\partial_t h_i^j=-\nabla_i\nabla^j f-f(h^2)_i^j+\nabla_\mathcal{T} h_i^j,\label{ev-Wg}\\
		&\partial_t H=-\Delta f-|A|^2 f+\nabla_\mathcal{T}{H},\label{ev-H}\\
		&\partial_t V=-\dot{V}_j^i\nabla_i\nabla^j f-f \dot{V}_j^i(h^2)_i^j+\nabla_\mathcal{T}{V},\,\,\text{for}\,\, V=V(h_i^j),\,\,\text{where}\,\, \dot{V}_j^i=\frac{\partial V}{\partial h_i^j}. \label{ev-F}
	\end{align}
\end{prop}

Using Proposition \ref{Prop-general-evo}, we obtain the following evolution equations along the flow \eqref{Inverse-flow}. For simplicity, we define the operator $\mathcal{L}$ as follows:
\begin{equation*}
	\mathcal{L}=\partial_t-\frac{1}{F^2}\dot{F}^{k\ell}\nabla_k\nabla_{\ell}-\langle \mathcal{T},\nabla \rangle.
\end{equation*}

\begin{prop}\label{Prop-evo-inverse}
	Along the inverse curvature flow \eqref{Inverse-flow}, the Weingarten matrix $h_i^j$, the curvature function $F$ and the mean curvature $H$ of $\Sigma_t$ evolve by
	\begin{align}
		(i)\mathcal{L}{h_i^j}&=\frac{1}{F^2}\ddot{F}^{k\ell,pq}\nabla_i{h_{k\ell}}\nabla^j{h_{pq}}-\frac{2}{F^3}\nabla_i{F}\nabla^j{F}+\frac{1}{F^2}\dot{F}^{k\ell}(h^2)_{k\ell}h_i^j-\frac{2}{F}(h^2)_i^j,\label{evo-hij}\\
		(ii)\mathcal{L}{F}&=-\frac{2}{F^3}\dot{F}^{ij}\nabla_i{F}\nabla_j{F}-\frac{1}{F}\dot{F}^{ij}(h^2)_{ij},\label{evo-F}\\
		(iii)\mathcal{L}{H}&=\frac{1}{F^2}\ddot{F}^{k\ell,pq}\nabla_i{h_{k\ell}}\nabla^i{h_{pq}}-\frac{2}{F^3}|\nabla F|^2+\frac{1}{F^2}\dot{F}^{k\ell}(h^2)_{k\ell}H-\frac{2}{F}|A|^2.\label{evo-H}
	\end{align}
Moreover, if $\Sigma_t$ is strictly convex, we assume that $\{b^{k\ell}\}$ is the inverse matrix of $\{h_{ij}\}$ and $\{b_m^n=b^{ns}g_{sm}\}$ is the inverse matrix of $\{h_i^j\}$. Then $b_m^n$ evolves by
\begin{align}
	(iv)\mathcal{L}{b_m^n}=&-\frac{1}{F^2}b_m^i b^{jn}\ddot{F}^{k\ell,pq}\nabla_i{h_{k\ell}}\nabla_j{h_{pq}}+\frac{2}{F^3}b_m^i b^{jn}\nabla_i F\nabla_j F\notag\\
	&-\frac{2}{F^2}b_m^pb^{in}b^{qs}\dot{F}^{k\ell}\nabla_k{h_{pq}}\nabla_{\ell}{h_{si}}-\frac{1}{F^2}\dot{F}^{k\ell}(h^2)_{k\ell}b_m^n+\frac{2}{F}\delta_m^n.\label{evo-b}
\end{align}
\end{prop}
\begin{proof}
	(i) Combining evolution equation \eqref{ev-Wg} with $f=\frac{1}{F}$,   we have
	\begin{align}
		\partial_t{h_i^j}&=-\nabla_i\nabla^j{\frac{1}{F}}-\frac{1}{F}(h^2)_i^j+\langle \mathcal{T},\nabla h_i^j\rangle\notag\\
		&=\frac{1}{F^2}\nabla_i\nabla^j{F}-\frac{2}{F^3}\nabla_i{F}\nabla^j{F}-\frac{1}{F}(h^2)_i^j+\langle \mathcal{T},\nabla h_i^j\rangle.\label{Ca-evo-1}
	\end{align}
By Ricci's identity, there holds
\begin{equation}\label{eq-Ric}
	\nabla_i\nabla_j h_{k\ell}=\nabla_{\ell}\nabla_{k} h_{ij}+h_{ij}(h^2)_{k\ell}-(h^2)_{ik}h_{j\ell}+(h^2)_{j\ell}h_{ki}-(h^2)_{ij}h_{k\ell}
\end{equation}
and hence
\begin{align}
	\nabla_i\nabla^j{F}&=\nabla_i(\dot{F}^{k\ell}\nabla^j{h_{k\ell}})\notag\\
	&=\ddot{F}^{k\ell,pq}\nabla_i{h_{pq}}\nabla^j{h_{k\ell}}+\dot{F}^{k\ell}\nabla_i\nabla^j{h_{k\ell}}\notag\\
	&=\ddot{F}^{k\ell,pq}\nabla_i{h_{k\ell}}\nabla^j{h_{pq}}+\dot{F}^{k\ell}\nabla_k\nabla_{\ell}h_i^j+\dot{F}^{k\ell}(h^2)_{k\ell}h_i^j-F(h^2)_i^j,\label{Ca-evo-2}
\end{align}
where we used \eqref{eq-Ric} and the fact that $\dot{F}^{k\ell}(h^2)_{ik}h_{\ell}^j=\dot{F}^{k\ell}(h^2)_{\ell}^jh_{ki}$. Combining \eqref{Ca-evo-1} and \eqref{Ca-evo-2} yields \eqref{evo-hij}.

(ii) Using \eqref{Ca-evo-1}, we calculate as follows:
\begin{align}
	\partial_t F&=\dot{F}_j^i\partial_t{h_i^j}\notag\\
	&=\frac{1}{F^2}\dot{F}^{ij}\nabla_i\nabla_j{F}-\frac{2}{F^3}\dot{F}^{ij}\nabla_i{F}\nabla_j{F}-\frac{1}{F}\dot{F}^{ij}(h^2)_{ij}+\langle \mathcal{T},\nabla F\rangle,
\end{align}
which gives \eqref{evo-F}.

(iii) \eqref{evo-H} follows from taking trace of \eqref{evo-hij} directly.

(iv) Note that 
\begin{align}
	\nabla_{\ell}{b_m^n}&=-b_i^n b_m^s\nabla_{\ell} h_s^i,\label{nabla-b}\\
	\nabla_k\nabla_{\ell}{b_m^n}&=-b_m^i b_j^n\nabla_k\nabla_{\ell}{h_i^j}+b_i^n b_m^p b_q^s\nabla_k{h_p^q}\nabla_{\ell}{h_s^i}+b_i^n b_m^p b_q^s\nabla_k{h_s^i}\nabla_{\ell}{h_p^q}\notag\\
	&=-b_m^i b_j^n\nabla_k\nabla_{\ell}{h_i^j}+b_m^p b^{in} b^{qs}\nabla_k{h_{pq}}\nabla_{\ell}{h_{si}}+b_m^p b^{in} b^{qs}\nabla_k{h_{si}}\nabla_{\ell}{h_{pq}}.\label{hessian-b}
\end{align}
Combining \eqref{evo-hij} with \eqref{nabla-b}, \eqref{hessian-b} gives
\begin{align*}
	\partial_t b_m^n=&-b_m^i b_j^n\partial_t{h_i^j}\\
	=&-\frac{1}{F^2}b_m^i b_j^n\dot{F}^{k\ell}\nabla_k\nabla_{\ell}{h_i^j}-b_m^i b_j^n\langle T,\nabla h_i^j \rangle-\frac{1}{F^2}b_m^i b^{jn}\ddot{F}^{k\ell,pq}\nabla_i{h_{k\ell}}\nabla_j{h_{pq}}\\
	&+\frac{2}{F^3}b_m^i b^{jn}\nabla_i F\nabla_j F-\frac{1}{F^2}\dot{F}^{k\ell}(h^2)_{k\ell}b_m^n+\frac{2}{F}\delta_m^n\\
	=&\frac{1}{F^2}\dot{F}^{k\ell}\nabla_k\nabla_{\ell}{b_m^n}+\langle T,\nabla{b_m^n} \rangle-\frac{1}{F^2}b_m^i b^{jn}\ddot{F}^{k\ell,pq}\nabla_i{h_{k\ell}}\nabla_j{h_{pq}}\\
	&-\frac{2}{F^2}b_m^p b^{in} b^{qs} \dot{F}^{k\ell}\nabla_k{h_{pq}}\nabla_{\ell}{h_{si}}+\frac{2}{F^3}b_m^i b^{jn}\nabla_i F\nabla_j F\\
	&-\frac{1}{F^2}\dot{F}^{k\ell}(h^2)_{k\ell}b_m^n+\frac{2}{F}\delta_m^n.
\end{align*}
This is \eqref{evo-b}.
\end{proof}
\section{Curvature estimates}\label{Sec-cur}
Assume that $\Sigma\subset\bar{\mathbb{B}}^{n+1} (n\geq 2)$ is a properly embedded, strictly convex smooth hypersurface with capillary boundary supported on $\mathbb{S}^n$ at a contact angle $\theta\in(0,\frac{\pi}{2}]$. Let $T^{*}>0$ be the maximal time such that the solution of the flow \eqref{Inverse-flow} starting from $\Sigma$ remains smooth on $[0,T^{*})$. Moreover, let $\bar{T}>0$ be the maximal time such that the solution remains both smooth and strictly convex for all $0\leq t<\bar{T}$. By definition, it follows that $\bar{T}\leq T^{*}\leq +\infty$. In this section, we prove that $\bar{T}=T^{*}<\infty$.

\begin{lem}\label{upper-F-H}
	Assume that $\Sigma\subset\bar{\mathbb{B}}^{n+1} (n\geq 2)$ be a properly embedded, strictly convex smooth hypersurface with capillary boundary supported on $\mathbb{S}^n$ at a contact angle $\theta\in(0,\frac{\pi}{2}]$. Let $\Sigma_t,t\in[0,\bar{T})$ be the solution of the flow \eqref{Inverse-flow} starting from $\Sigma$. Then on the time interval $[0,\bar{T})$, we have
	\begin{equation}
		\max_{\bar{\mathbb{B}}^n} F(\cdot,t)\leq \max_{\bar{\mathbb{B}}^n} F(\cdot,0)\,\qquad \max_{\bar{\mathbb{B}}^n} H(\cdot,t)\leq \max_{\bar{\mathbb{B}}^n} H(\cdot,0).
	\end{equation}
\end{lem}
\begin{proof}
	Since $\Sigma_t$ is a strictly convex solution on $[0,\bar{T})$, by \eqref{de-boundary}, we see that along $\partial\Sigma_t$,
	\begin{equation}\label{de-F}
		\nabla_{\mu} F=-(\frac{1}{\sin\theta}+\cot\theta h_{\mu\mu})F<0, \ \text{for}\ t\in(0,\bar{T}). 
	\end{equation}
Then the maximum of $F$ on $\Sigma_t$ is attained on the interior point. By \eqref{evo-F}, we have
\begin{equation}
	\partial_t F_{\max}(t)\leq 0,
\end{equation}
which means $\max_{\bar{\mathbb{B}}^n} F(\cdot,t)\leq \max_{\bar{\mathbb{B}}^n} F(\cdot,0)$.

Moreover, let $\{e_{\alpha}\}_{\alpha=1}^{n-1}$ be an orthonormal frame of $T(\partial\Sigma_t)$ such that $\{\mu,e_1,\cdots,e_{n-1}\}$ forms an orthonormal frame of $T(\Sigma_t)$. Then combining with \eqref{de-F}, we obtain
\begin{equation*}
	-(\frac{1}{\sin\theta}+\cot\theta h_{\mu\mu})F=\nabla_{\mu}F=\dot{F}^{\mu\mu}\nabla_{\mu}h_{\mu\mu}+\sum_{\alpha=1}^{n-1}\dot{F}^{\alpha\alpha}\nabla_{\mu}{h_{\alpha\alpha}},
\end{equation*}
which yields
\begin{equation}\label{nablahmumu}
	\nabla_{\mu}{h_{\mu\mu}}=-(\frac{1}{\sin\theta}+\cot\theta h_{\mu\mu})\frac{F}{\dot{F}^{\mu\mu}}-\sum_{\alpha=1}^{n-1}\frac{\dot{F}^{\alpha\alpha}}{\dot{F}^{\mu\mu}}\nabla_{\mu}{h_{\alpha\alpha}}.
\end{equation}
Hence
\begin{align}
	\nabla_{\mu}H=&\sum_{\alpha=1}^{n-1}\nabla_{\mu}{h_{\alpha\alpha}}+\nabla_{\mu}{h_{\mu\mu}}\notag\\
	=&-(\frac{1}{\sin\theta}+\cot\theta h_{\mu\mu})\frac{F}{\dot{F}^{\mu\mu}}+\sum_{\alpha=1}^{n-1}\frac{1}{\dot{F}^{\mu\mu}}(\dot{F}^{\mu\mu}-\dot{F}^{\alpha\alpha})\nabla_{\mu}h_{\alpha\alpha}\notag\\
	=&-(\frac{1}{\sin\theta}+\cot\theta h_{\mu\mu})\frac{F}{\dot{F}^{\mu\mu}}+\sum_{\alpha=1}^{n-1}\frac{1}{\dot{F}^{\mu\mu}}(\dot{F}^{\mu\mu}-\dot{F}^{\alpha\alpha})(h_{\mu\mu}-h_{\alpha\alpha})\tilde{h}_{\alpha\alpha},
\end{align}
where in the last equality we used the property $(4)$ in Proposition \ref{boundary-h-property}. Since $F$ is a concave function,  by Lemma \ref{lem-concave} we have $(\dot{F}^{\mu\mu}-\dot{F}^{\alpha\alpha})(h_{\mu\mu}-h_{\alpha\alpha})\leq 0$, which yields
\begin{equation}
	\nabla_{\mu} H<0 \quad \text{along}\,\,\,\, \partial\Sigma_t.
\end{equation}
Thus, the maximum of $H$ on $\Sigma_t$ is attained at the interior point. By the evolution equation \eqref{evo-H} of $H$,  at the maximum point of $H$, we obtain
\begin{equation}\label{p-Hmax}
	\partial_t H_{\max}(t)\leq \frac{1}{F^2}\dot{F}^{k\ell}(h^2)_{k\ell}H-\frac{2}{F}|A|^2.
\end{equation}
Since
\begin{align}
	\dot{F}^{k\ell}(h^2)_{k\ell}H-F|A|^2\notag
	=&\sum_{i,j=1}^n(\frac{\partial F}{\partial\kappa_i}\kappa_i^2\kappa_j-\frac{\partial F}{\partial\kappa_i}\kappa_i\kappa_j^2)\notag\\
	=&\sum_{1\leq i<j\leq n}\kappa_i\kappa_j(\kappa_i-\kappa_j)(\frac{\partial F}{\partial\kappa_i}-\frac{\partial F}{\partial\kappa_j})\notag\\
	\leq &0,\label{Ca-con}
\end{align}
where in the last inequality we used Lemma \ref{lem-concave} again. Combining \eqref{p-Hmax} with \eqref{Ca-con}, we have
\begin{equation}
	\partial_t H_{\max}(t)\leq -\frac{1}{F}|A|^2<0,
\end{equation}
which yields $\max_{\bar{\mathbb{B}}^n} H(\cdot,t)\leq \max_{\bar{\mathbb{B}}^n} H(\cdot,0)$.
\end{proof}
According to Lemma \ref{upper-F-H}, the principal curvatures of $\Sigma_t$ on the time  interval $[0,\bar{T})$ satisfies 
\begin{equation}\label{uper-H}
	0<\kappa_i<H(\cdot,t)\leq \max_{\bar{\mathbb{B}}^n} H(\cdot,0).
\end{equation}
Let $y(\cdot,t):\partial\mathbb{B}^n\to\mathbb{S}^n$ be the induced embedding of $x(\cdot,t):\bar{\mathbb{B}}^{n}\to\bar{\mathbb{B}}^{n+1}$. By property $(2)$ in Proposition \ref{boundary-h-property}, it follows that $\partial\Sigma_t$ is a strictly convex hypersurface in $\mathbb{S}^n$ for all time $t\in[0,\bar{T})$. Combining this with \eqref{uper-H}, we  obtain uniform $C^2$ estimates of $\partial\Sigma_t\subset\mathbb{S}^n$. Then there exists a $C^{1,\alpha}$ limiting hypersurface $\partial\Sigma_{\bar{T}}$, such that $\partial\Sigma_t$ converges to $\partial\Sigma_{\bar{T}}$ in the norm of $C^{1,\alpha}$. Moreover, we have the following rigidity result.

\begin{lem}\label{lem-rigidity}
	There exists a vector $e\in \mathrm{int}(\widehat{\partial\Sigma_{\bar{T}}})$, such that $\partial\Sigma_{\bar{T}}$ either is $\partial C_{\theta,\infty}(e)$, or $\partial\Sigma_{\bar{T}}\subset\mathrm{int}(\widehat{\partial C_{\theta,\infty}}(e))$.
\end{lem}
\begin{proof}
	The proof is essentially the same as that of \cite[Lemma A.1]{HWYZ-2023}, except that in Step 1 of  \cite[Lemma A.1]{HWYZ-2023}, an additional approximation argument is required using $\partial\Sigma_t$ as $t\to\bar{T}$. Therefore, we omit the details here.
\end{proof}
In the following, we focus on showing that $\bar{T}=T^{*}$. Firstly, we need a further estimate on the curvature function $F$.

\begin{lem}\label{lowbound-F}
	Assume that $\Sigma\subset\bar{\mathbb{B}}^{n+1} (n\geq 2)$ is a properly embedded, strictly convex smooth hypersurface with capillary boundary supported on $\mathbb{S}^n$ at a contact angle $\theta\in(0,\frac{\pi}{2}]$. Let $\Sigma_t,t\in[0,\bar{T})$ be the solution of the flow \eqref{Inverse-flow} starting from $\Sigma$. Assume that $\partial\Sigma_{\bar{T}}$ is not a flat ball around some point $e\in\mathbb{S}^n$, then there holds
	\begin{equation}\label{lower-F}
		\max_{[0,\bar{T})\times\mathbb{B}^n}{\frac{1}{F}}\leq c,
	\end{equation}
where $c>0$ is a constant depending on $\Sigma$ and the distance of $\partial\Sigma_{\bar{T}}$ to a suitable geodesic slice $\partial C_{\theta,\infty}(e)$.
\end{lem}
\begin{proof}
By Lemma \ref{lem-rigidity}, let $e\in \mathrm{int}(\widehat{\partial\Sigma_{\bar{T}}})$ be a vector in $\mathbb{S}^{n}$ such that $\partial\Sigma_{\bar{T}}\subset\mathrm{int}(\widehat{\partial C_{\theta,\infty}}(e))$. Then for $t$ sufficiently close to $\bar{T}$, we have
\begin{equation*}
	e\in \mathrm{int}(\widehat{\partial\Sigma_{{t}}}),
\end{equation*}
then we can apply Lemma \ref{Lem-es-height} to obtain a uniform lower bound for the height function $\langle x,e\rangle$, such that 
\begin{equation}\label{lower-height}
	\langle x,e\rangle\geq \cos\theta+\delta,\quad \forall x\in\Sigma_t,t\in[0,\bar{T}).
\end{equation}
We define the function $\chi(x)$ as follows:
\begin{equation}
	\chi(x)=\frac{1}{2}|x|^2+\frac{\beta-1}{2}(\langle x,e\rangle^2-\cos^2{\theta})-\lambda(\langle x,e\rangle-\cos\theta)+1,
\end{equation}
where 
\begin{equation*}
	\lambda>\frac{1}{\delta},\,\,0<\beta<1,\,\,0<\delta<1-\cos\theta.
\end{equation*}
Then, by direct computation, we obtain
\begin{equation}\label{uni-chi}
	\frac{\beta}{2}-\lambda+\frac{1-\beta}{2}\cos^2{\theta}+\lambda\cos\theta+1\leq \chi(x)\leq \frac{1}{2}+\frac{\beta-1}{2}\delta^2< \frac{1}{2}.
\end{equation}
Next, we define the auxiliary function $\xi$ by
\begin{equation}
	\xi=\frac{1}{F(1+\cos\theta\langle x,\nu \rangle)}\frac{1}{\frac{1}{2}-\chi}:=\frac{G(\chi)}{F(1+\cos\theta\langle x,\nu \rangle)},
\end{equation}
The estimate in \eqref{uni-chi} ensures that the function $\xi$ is well defined. Moreover,  there exists uniform constants $0<c_1<c_2<+\infty$ such that 
\begin{align}\label{uni-G}
c_1\leq	\frac{G(\chi)}{1+\cos\theta\langle x,\nu \rangle}\leq c_2.
\end{align}
We now compute $\nabla_{\mu}{\xi}$ along $\partial\Sigma_t$:
\begin{align}
	\nabla_{\mu}(1+\cos\theta\langle x,\nu \rangle)&=\cos\theta\langle\nabla_{\mu}{x},\nu \rangle+\cos\theta\langle x,\nabla_{\mu}{\nu}\rangle\notag\\
	&=\cos\theta\langle \mu,\nu\rangle+\cos\theta h_{\mu\mu}\langle x,\mu \rangle\notag\\
	&=\cos\theta h_{\mu\mu}\langle \sin\theta\mu-\cos\theta\nu,\mu \rangle\notag\\
	&=\cos\theta\sin\theta h_{\mu\mu}\notag\\
	&=\cot\theta h_{\mu\mu}(1+\cos\theta\langle x,\nu \rangle),\label{db-1}\\
	\nabla_{\mu}{\frac{1}{F}}&=(\frac{1}{\sin\theta}+\cot\theta h_{\mu\mu})\frac{1}{F},\label{db-2}\\
	\nabla_{\mu}{G}&=G'\nabla_{\mu}{\chi}=G^2\nabla_{\mu}{\chi}\notag\\
	&=G^2\left[\langle x,\mu\rangle+(\beta-1)\langle x,e\rangle\langle \mu,e\rangle-\lambda\langle\mu,e\rangle\right].\label{db-3}
\end{align}
Combining \eqref{db-1} with \eqref{db-2} and \eqref{db-3}, we have
\begin{align}
	\nabla_{\mu}{\xi}
	=&\left(\nabla_{\mu}{\frac{1}{F}}\right)\frac{G(\chi)}{1+\cos\theta\langle x,\nu \rangle}+\left(\nabla_{\mu}G\right)\frac{1}{F(1+\cos\theta\langle x,\nu \rangle)}+\frac{G}{F}\nabla_{\mu}\left(\frac{1}{1+\cos\theta\langle x,\nu \rangle}\right)\notag\\
	=&\xi\left(\frac{1}{\sin\theta}+G\nabla_{\mu}{\chi}\right)\notag\\
	=&\frac{\xi G}{\sin\theta}\left[\frac{\beta-1}{2}\langle x+\cos\theta\nu,e\rangle^2+\frac{\beta-1}{2}\cos^2{\theta}|e^{\top}|^2-\cos^2\theta-\lambda\cos\theta(1+\langle\nu,e\rangle)\right]<0,\label{boundary-xi}
\end{align}
and hence the maximum of $\xi$ is attained at an interior point of $\Sigma_t$. Next, we compute the evolution equation of $\xi$. We start with the evolution equation of $\chi$ as follows:
\begin{align}
	\partial_t\chi=&\langle\partial_t{x},x\rangle+(\beta-1)\langle x,e\rangle\langle\partial_t{x},e\rangle-\lambda\langle \partial_t{x},e\rangle\notag\\
	=&\frac{1}{F}\langle x,\nu\rangle+\frac{\beta-1}{F}\langle x,e\rangle\langle\nu,e\rangle-\frac{\lambda}{F}\langle\nu,e\rangle+\langle x,\mathcal{T}\rangle+(\beta-1)\langle x,e\rangle\langle \mathcal{T},e\rangle-\lambda\langle \mathcal{T},e\rangle\notag\\
	=&\frac{1}{F}\langle x,\nu\rangle+\frac{\beta-1}{F}\langle x,e\rangle\langle\nu,e\rangle-\frac{\lambda}{F}\langle\nu,e\rangle+\langle \mathcal{T},\nabla\chi \rangle, \label{partial-chi}
\end{align}
\begin{align}
		\frac{1}{F^2}\dot{F}^{k\ell}\nabla_k\nabla_{\ell}{\chi}
		=&\frac{1}{F^2}\dot{F}^{k\ell}\big(g_{k\ell}+\langle x,x_{;\ell k}\rangle+(\beta-1)\langle e_k,e\rangle\langle e_{\ell},e\rangle+(\beta-1)\langle x,e\rangle\langle x_{;\ell k},e\rangle-\lambda\langle x_{;\ell k},e\rangle \big)\notag\\
	=&\frac{1}{F^2}\dot{F}^{ij}g_{ij}-\frac{1}{F}\langle x,\nu\rangle+\frac{\beta-1}{F^2}\dot{F}^{ij}\langle e_i,e\rangle\langle e_j,e\rangle-\frac{\beta-1}{F}\langle x,e\rangle\langle \nu,e\rangle+\frac{\lambda}{F}\langle \nu,e\rangle.\label{hessian-chi}
\end{align}
Hence, we have
\begin{align}
	\mathcal{L}{\chi}=&\partial_t\chi-\frac{1}{F^2}\dot{F}^{k\ell}\nabla_k\nabla_{\ell}{\chi}-\langle \mathcal{T},\nabla\chi \rangle\notag\\
	=&\frac{2}{F}\langle x,\nu\rangle+\frac{2(\beta-1)}{F}\langle x,e\rangle\langle \nu,e\rangle-\frac{2\lambda}{F}\langle \nu,e\rangle-\frac{\beta-1}{F^2}\dot{F}^{ij}\langle e_i,e\rangle\langle e_j,e\rangle-\frac{1}{F^2}\dot{F}^{ij}g_{ij}.\label{L-chi}
\end{align}
Then, by the definition of $G$, we obtain the following evolution equation:
\begin{align}
	\mathcal{L}G=&G^2\mathcal{L}{\chi}-\frac{2G^3}{F^2}\dot{F}^{ij}\nabla_i{\chi}\nabla_j{\chi}\notag\\
	=&\frac{2G^2}{F}\langle x,\nu\rangle+\frac{2(\beta-1)G^2}{F}\langle x,e\rangle\langle \nu,e\rangle-\frac{2\lambda G^2}{F}\langle \nu,e\rangle-\frac{(\beta-1)G^2}{F^2}\dot{F}^{ij}\langle e_i,e\rangle\langle e_j,e\rangle\notag\\
	&-\frac{G^2}{F^2}\dot{F}^{ij}g_{ij}-\frac{2G^3}{F^2}\dot{F}^{ij}\nabla_i{\chi}\nabla_j{\chi}.\label{L-G}
\end{align}
On the other hand, applying \eqref{ev-F} with $V=\frac{1}{F}$, we obtain
\begin{equation}\label{L-F}
	\mathcal{L}\left({\frac{1}{F}}\right)=\frac{1}{F^3}\dot{F}^{ij}(h^2)_{ij}.
\end{equation}
Moreover, using \eqref{ev-normal} we compute
\begin{align}
	\partial_t\left(\frac{1}{1+\cos\theta\langle x,\nu\rangle}\right)=&-\frac{\cos\theta}{(1+\cos\theta\langle x,\nu\rangle)^2}\partial_t{\langle x,\nu\rangle}\notag\\
	=&-\frac{\cos\theta}{(1+\cos\theta\langle x,\nu\rangle)^2}\left(\frac{1}{F}+\langle x,-\nabla \frac{1}{F}+g^{ij}h(e_i,\mathcal{T})e_j\rangle\right)\notag\\
	=&-\frac{\cos\theta}{(1+\cos\theta\langle x,\nu\rangle)^2}\frac{1}{F}+\frac{\cos\theta}{(1+\cos\theta\langle x,\nu\rangle)^2}\langle x,\nabla \frac{1}{F}\rangle\notag\\
	&+\langle \mathcal{T},\nabla\left(\frac{1}{1+\cos\theta\langle x,\nu\rangle}\right)\rangle,\label{partail-xnu}
\end{align}
where in the second equality we used \eqref{ev-normal}.
We also have
\begin{align}
	\nabla_j\left(\frac{1}{1+\cos\theta\langle x,\nu\rangle}\right)=&-\frac{\cos\theta}{(1+\cos\theta\langle x,\nu\rangle)^2}h_j^k\langle x,e_k\rangle,\label{n-xnu}\\
	\nabla_i\nabla_j\left(\frac{1}{1+\cos\theta\langle x,\nu\rangle}\right)=&\frac{2\cos^2\theta}{(1+\cos\theta\langle x,\nu\rangle)^3}h_i^{\ell}h_j^k \langle x,e_k\rangle\langle x,e_{\ell}\rangle-\frac{\cos\theta}{(1+\cos\theta\langle x,\nu\rangle)^2}h_{j;i}^k\langle x,e_k\rangle\notag\\
	&-\frac{\cos\theta}{(1+\cos\theta\langle x,\nu\rangle)^2}h_{ij}+\frac{\cos\theta}{(1+\cos\theta\langle x,\nu\rangle)^2}\langle x,\nu\rangle(h^2)_{ij}.\label{h-xnu}
\end{align}
Combining \eqref{partail-xnu} with \eqref{h-xnu} yields
\begin{equation}
	\mathcal{L}\left(\frac{1}{1+\cos\theta\langle x,\nu\rangle}\right)=-\frac{2\cos^2\theta}{F^2(1+\cos\theta\langle x,\nu\rangle)^3}\dot{F}^{ij}h_i^{\ell}h_j^k \langle x,e_k\rangle\langle x,e_{\ell}\rangle-\frac{\cos\theta\langle x,\nu\rangle
	}{F^2(1+\cos\theta\langle x,\nu\rangle)^2}\dot{F}^{ij}(h^2)_{ij}.\label{L1xnu}
\end{equation}
Finally,  the evolution of the auxiliary function of $\xi$ satisfies
\begin{align}
	\mathcal{L}{\xi}=&\frac{\mathcal{L}G}{F(1+\cos\theta\langle x,\nu \rangle)}+\frac{G}{1+\cos\theta\langle x,\nu \rangle}\mathcal{L}\left(\frac{1}{F}\right)+\frac{G}{F}\mathcal{L}\left(\frac{1}{1+\cos\theta\langle x,\nu \rangle}\right)\notag\\
	&-2\frac{\dot{F}^{ij}}{F^3}\nabla_i\left(\frac{1}{1+\cos\theta\langle x,\nu \rangle}\right)\nabla_j{G}-\frac{2}{F^2(1+\cos\theta\langle x,\nu \rangle)}\dot{F}^{ij}\nabla_i\left(\frac{1}{F}\right)\nabla_j G\notag\\
	&-\frac{2G}{F^2}\dot{F}^{ij}\nabla_i\left(\frac{1}{F}\right)\nabla_j\left(\frac{1}{1+\cos\theta\langle x,\nu \rangle}\right).\label{Lxi-1}
\end{align}
Since the maximum of $\xi$ is attained at an interior point, we assume that 
\begin{align*}
	\xi(x,t)=\max_{\Sigma\times[0,t]}\xi, \ t\in[0,\bar{T}).
\end{align*} 
Then at this  point, we have $\mathcal{L}\xi\geq 0$ and
\begin{equation*}
	\nabla_i\left(\frac{G}{F(1+\cos\theta\langle x,\nu \rangle)}\right)=0,
\end{equation*}
which yields
\begin{align}
	\nabla_i\left(\frac{1}{F}\right)=&-\frac{\nabla_i G}{GF}-\frac{1+\cos\theta\langle x,\nu \rangle}{F}\nabla_i\left(\frac{1}{1+\cos\theta\langle x,\nu \rangle}\right)\notag\\
	=&-\frac{G}{F}\nabla_i{\chi}-\frac{1+\cos\theta\langle x,\nu \rangle}{F}\nabla_i\left(\frac{1}{1+\cos\theta\langle x,\nu \rangle}\right),\label{n-1-F}
\end{align}
where 
\begin{align}
	\nabla_i{\chi}=&\langle x,e_i\rangle+(\beta-1)\langle x,e\rangle \langle e_i,e\rangle-\lambda \langle e_i,e\rangle,\label{nabla-chi}\\
	\nabla_i G=&G'\nabla_i\chi=G^2\nabla_i\chi.\label{nabla-G}
\end{align}
Combining \eqref{n-1-F}, \eqref{nabla-chi} with \eqref{n-xnu}, we obtain
\begin{align}
	&\mathrm{the\,\,sum \,\,of\,\, the\,\, last\,\, two \,\,lines\,\, of\,\,\eqref{Lxi-1}}\nonumber\\
 =&2\xi^3(1+\cos\theta\langle x,\nu \rangle)^2\dot{F}^{ij}\nabla_i\chi \nabla_j \chi-\frac{2\cos\theta}{F}\dot{F}^{ij}h_j^k\nabla_i{\chi}\langle x,e_k\rangle\xi^2\notag\\
	&+\frac{2\cos^2\theta}{G^2}\dot{F}^{ij}h_j^k h_i^{\ell}\langle x,e_k\rangle\langle x,e_{\ell}\rangle\xi^3.\label{Lxi-2}
\end{align}
Now, combining \eqref{L-G}, \eqref{L-F}, \eqref{L1xnu}, \eqref{Lxi-1} with \eqref{Lxi-2}, at the point $(x,t)$ we have
\begin{align}
	0\leq\mathcal{L}{\xi}(x,t)=&2\langle x,\nu\rangle(1+\cos\theta\langle x,\nu \rangle)\xi^2+2(\beta-1)\langle x,e\rangle\langle\nu,e \rangle(1+\cos\theta\langle x,\nu \rangle)\xi^2\notag\\
	&-2\lambda\langle\nu,e \rangle(1+\cos\theta\langle x,\nu \rangle)\xi^2-\frac{1}{F}\dot{F}^{ij}g_{ij}(1+\cos\theta\langle x,\nu \rangle)\xi^2\notag\\
	&-\frac{\beta-1}{F}\dot{F}^{ij}\langle e_i,e\rangle \langle e_j,e\rangle(1+\cos\theta\langle x,\nu \rangle)\xi^2+\frac{\dot{F}^{ij}(h^2)_{ij}\xi}{F^2(1+\cos\theta\langle x,\nu \rangle)}\notag\\
	&-\frac{2}{F}\cos\theta\dot{F}^{ij}h_j^k\nabla_i \chi \langle x,e_k\rangle\xi^2.\label{ev-xi}
\end{align}
Note that the geometric quantities, such as $\langle x,\nu\rangle$, $\langle \nu,e\rangle$, $\langle x,e\rangle$ and $|\nabla\chi|$ are uniformly bounded. Moreover, 
\begin{align*}
	\dot{F}^{ij}\langle e_i,e\rangle\langle e_j,e\rangle=&\sum_{i=1}^n\frac{\partial F}{\partial\kappa_i}\langle e_i,e\rangle^2\leq \sum_{i=1}^n\frac{\partial F}{\partial\kappa_i}=\dot{F}^{ij}g_{ij},\\
	\dot{F}^{ij}(h^2)_{ij}=&\sum_{i=1}^n{\frac{\partial F}{\partial \kappa_i}\kappa_i^2}\leq H\sum_{i=1}^n\frac{\partial F}{\partial\kappa_i}\kappa_i=FH\leq CF,
\end{align*}
from which we conclude
\begin{align}
	&-\frac{1}{F}\dot{F}^{ij}g_{ij}(1+\cos\theta\langle x,\nu \rangle)\xi^2-\frac{\beta-1}{F}\dot{F}^{ij}\langle e_i,e\rangle \langle e_j,e\rangle(1+\cos\theta\langle x,\nu \rangle)\xi^2\notag\\
	\leq&-\frac{\beta}{F}\dot{F}^{ij}g_{ij}(1+\cos\theta\langle x,\nu \rangle)\xi^2\notag\\
	\leq&-\frac{\beta}{F}(1-\cos\theta)\xi^2,\label{es-1}
\end{align}
and
\begin{align}
	\frac{\dot{F}^{ij}(h^2)_{ij}\xi}{F^2(1+\cos\theta\langle x,\nu \rangle)}\leq \frac{C}{F}\xi,\label{es-2}
\end{align}
where in the last inequality of \eqref{es-1} we used the fact that $\dot{F}^{ij}g_{ij}\geq 1$. We estimate the last line of \eqref{ev-xi} as follows:
\begin{align}
	&-\frac{2}{F}\cos\theta\dot{F}^{ij}h_j^k\nabla_i \chi \langle x,e_k\rangle\xi^2\notag\\
	=&-\frac{2}{F}\cos\theta\xi^2\sum_{i=1}^n{\frac{\partial F}{\partial\kappa_i}\kappa_i}\nabla_i\chi \langle x,e_i\rangle\notag\\
	\leq&\frac{C\xi^2}{F}\sum_{i=1}^n{\frac{\partial F}{\partial\kappa_i}\kappa_i}
	=C\xi^2.\label{es-3}
\end{align}
Combining \eqref{ev-xi}, \eqref{es-1}, \eqref{es-2} with \eqref{es-3}, there exists a uniform constant $C>0$ such that 

\begin{equation}\label{In-ximax}
 \frac{C}{F}\xi(x,t)+(C-(1-\cos\theta)\frac{\beta}{F})\xi^2(x,t)\geq0.
\end{equation}
By \eqref{uni-G} and the definition of $\xi$, we obtain
\begin{align}\label{est-xi-2}
\frac{1}{F}\leq\frac{\xi}{c_1}, \quad \text{and}\ 	(1-\cos\theta)\frac{\beta}{F}\geq\frac{\beta(1-\cos\theta)}{c_2}\xi.
\end{align}
Substituting \eqref{est-xi-2} into \eqref{In-ximax}, at the maximum point $(x,t)$, we now have
\begin{align}
	\left(C\frac{c_1+1}{c_1}-\frac{\beta(1-\cos\theta)}{c_2}\xi\right)\xi^2\geq0,
\end{align}
this implies that 
\begin{align*}
	\xi(x,t)\leq C\frac{c_2(c_1+1)}{c_1\beta(1-\cos\theta)}.
\end{align*}
	Combining this with the uniform bound \eqref{uni-G} for $\frac{G}{1+\cos\theta\langle x,\nu\rangle}$ gives an uniform upper bound for $\frac{1}{F}$.

%
%

\end{proof}
Next, using the upper bound of the second fundamental form $\vert A\vert^2$ and the lower bound of the function $F$, we conclude that $F$ is uniform elliptic along the flow.
\begin{lem}\label{uniform elliptic}
	 Assume that $\partial\Sigma_{\bar{T}}$ is not a flat ball around some point $e\in\mathbb{S}^n$, then for all $ i=1,\cdots,n$ there holds
	\begin{align*}
		0<\epsilon\leq \frac{\partial F}{\partial\kappa_i}(z,t)\leq C<+\infty, \ \forall  (z,t)\in \mathbb{B}^n\times[0,\bar{T}),
	\end{align*} 
	where $\epsilon$ and $C$ are two uniform positive  constants  depending on $\Sigma$ and the distance of $\partial\Sigma_{\bar{T}}$ to a suitable geodesic slice $\partial C_{\theta,\infty}(e)$.
\end{lem}
\begin{proof}
	By Lemma \ref{upper-F-H},  there exists a positive constant $R$ such that $\vert A\vert\leq R$ on $\Sigma_t$ for all $t\in[0,\bar{T})$. This implies that the principal curvatures $(\kappa_i)$ of $\Sigma_t $ $(t\in[0,\bar{T}))$ lie in a compact subset $\Lambda\subset\subset\Gamma\cap B_{2R}$. Otherwise, there would exist a sequence of points $p_k$ such that $(\kappa_1,\cdots,\kappa_n)(p_k)$ tends to $\partial\Gamma$ as $k\to \infty$. Since $F$ vanishes on $\partial\Gamma$ by Assumption \ref{assum}, we have $F(p_k)\to 0$ as $k\to\infty$, which contradicts the uniform lower bound $\frac{1}{F}\leq c$ given by Lemma \ref{lowbound-F}. Therefore, the principal curvatures remain uniformly away from $\partial\Gamma$, and since we assume $\frac{\partial F}{\partial\kappa_i}>0$, it follows that
	 $\frac{\partial F}{\partial\kappa_i}$ is uniformly bounded on $\Lambda$. This completes the proof.
\end{proof}

Now, we are prepared to show that $\bar{T}=T^{*}$.
\begin{lem}\label{Lem-pconvex}
	As long as the flow \eqref{Inverse-flow} exists, $\Sigma_t$ remains strictly convex provided the initial hypersurface $\Sigma$ is strictly convex.
\end{lem}
\begin{proof}
If there exists a vector $e\in\mathbb{S}^n$, such that $\partial\Sigma_{\bar{T}}=\partial C_{\theta,\infty}(e)$, then $\partial\Sigma_{\bar{T}}$ is the singularity of the flow \eqref{Inverse-flow} and hence $\bar{T}=T^{*}$. In the following, we prove that if $\partial\Sigma_{\bar{T}}$ is not a flat ball around some point $e\in\mathbb{S}^n$, then the principal curvatures of $\Sigma_t$ have an uniform positive lower bound as $t\to\bar{T}$, which contradicts the definition of $\bar{T}$. By Lemma \ref{lowbound-F}, if $\partial\Sigma_{\bar{T}}$ is not a flat ball around some point $e\in\mathbb{S}^n$, then there exits a constant $c_1>0$, such that  
\begin{equation*}
	F(\cdot,t)\geq c_1,\quad \text{in}\,\,\,\bar{\mathbb{B}}^n\times[0,\bar{T}).
\end{equation*}
By \eqref{re-F-H}, we have
\begin{equation}\label{lowbound-H}
	H(\cdot,t)\geq nF(\cdot,t)\geq  nc_1:=c_2,\quad \text{in}\,\,\,\bar{\mathbb{B}}^n\times[0,\bar{T}).
\end{equation}
We define 
\begin{equation}
	\tilde{H}=\sum_{i=1}^n{\frac{1}{\kappa_i}}=g_{ij}b^{ij},
\end{equation}
where $\{b^{ij}\}$ is the inverse matrix of $\{h_{ij}\}$, which is defined in Proposition \ref{Prop-evo-inverse}. Along $\partial\Sigma_t$,  since $\mu$ is a principal direction, we may choose  a local orthonormal basis  $\{\mu,e_{\alpha}\},\alpha=1,\cdots,n-1$  on $\Sigma_t$, such that $b^{\mu\alpha}=0$ for all $1\leq\alpha\leq n-1$. 

We now compute $\nabla_{\mu}\tilde{H}$ along $\partial\Sigma_t$. Using  Property $(3)$ of Proposition \ref{boundary-h-property} and \eqref{nablahmumu} we obtain
\begin{align}
	\nabla_{\mu}{\tilde{H}}=&\nabla_{\mu}{b^{\mu\mu}}+\sum_{\alpha=1}^n{\nabla_{\mu}b^{\alpha\alpha}}\notag\\
	=&-(b^{\mu\mu})^2\nabla_{\mu}{h_{\mu\mu}}-(b^{\alpha\alpha})^2\nabla_{\mu}{h_{\alpha\alpha}}\notag\\
	=&(b^{\mu\mu})^2\left(\frac{1}{\sin\theta}+\cot\theta h_{\mu\mu}\right)\frac{F}{\dot{F}^{\mu\mu}}+(b^{\mu\mu})^2\frac{\dot{F}^{\alpha\alpha}}{\dot{F}^{\mu\mu}}\nabla_{\mu}{h_{\alpha\alpha}}-(b^{\alpha\alpha})^2\nabla_{\mu}{h_{\alpha\alpha}}\nonumber\\
	=&\left(\frac{1}{\sin\theta}+\cot\theta h_{\mu\mu}\right)\left(1+\frac{\dot{F}^{\alpha\alpha}}{\dot{F}^{\mu\mu}}\right)b^{\mu\mu}-\cot\theta(b^{\mu\mu})^2(h_{\alpha\alpha}-h_{\mu\mu})^2\frac{\dot{F}^{\alpha\alpha}}{\dot{F}^{\mu\mu}}\nonumber\\
	&-(b^{\alpha\alpha})^2\left(\frac{1}{\sin\theta}+\cot\theta h_{\alpha\alpha}\right)(h_{\mu\mu}-h_{\alpha\alpha})\nonumber\\
	\leq&\left(\frac{1}{\sin\theta}+\cot\theta h_{\mu\mu}\right)\left(1+\frac{\dot{F}^{\alpha\alpha}}{\dot{F}^{\mu\mu}}\right)b^{\mu\mu}+\left(\frac{1}{\sin\theta}+\cot\theta h_{\alpha\alpha}\right)b^{\alpha\alpha}.
\end{align} 
By Lemma \ref{uniform elliptic}, there exist two positive constants $\epsilon$ and $C$, such that 
\begin{equation*}
	\epsilon\leq \dot{F}^{\mu\mu},\dot{F}^{\alpha\alpha}\leq C,\quad \alpha=1,\cdots,n-1.
\end{equation*}
It then follows that there exists a positive constant $c_3>0$ such that
\begin{equation}\label{ntH}
	\nabla_{\mu}\tilde{H}< c_3\tilde{H},
\end{equation}
since $h_{\mu\mu}<H$ and $b^{\mu\mu}<\tilde{H}$. Define the auxiliary function as
\begin{equation}\label{fun-phi}
	\phi=\log\tilde{H}-\beta\log{\langle x,e \rangle}-\alpha t, 
\end{equation}
where $\alpha,\beta$ are two positive constants to be determined. Since $\partial\Sigma_{\bar{T}}\subset\mathrm{int}(\widehat{\partial C_{\theta,\infty}}(e))$,  if we denote by $\bar{B}_r(e)$ the geodesic ball in $\mathbb{S}^n$ centered at $e$ with radius $r$, then there exists a constant $r_0<\theta$,   depending on the distance between $\partial\Sigma_{\bar{T}}$ and $\partial C_{\theta,\infty}(e)$, such that $\widehat{\partial\Sigma}\subset\bar{B}_{r_0}(e)$. Consequently, along $\partial\Sigma_t$ for all $t\in[0,\bar{T})$, we have the following estimates:
\begin{equation*}
	\langle x,e\rangle\geq \cos r_0,\quad -\sin r_0\leq \langle \bar{\nu},e \rangle\leq 0.
\end{equation*}
It then follows that
\begin{align}
	\nabla_{\mu}\langle x,e\rangle=&\langle \mu,e\rangle\notag\\
	=&\sin\theta\langle x,e\rangle+\cos\theta\langle \bar{\nu},e \rangle\notag\\
	\geq& \sin\theta\cos r_0-\cos\theta\sin r_0\notag\\
	=&\sin(\theta-r_0)\notag\\
	\geq&c_4\geq c_4\langle x,e\rangle,\quad \text{along}\,\,\partial\Sigma_t.
\end{align}
Therefore, if we choose
\begin{equation}\label{defn-beta}
	\beta=\frac{c_3+1}{c_4},
\end{equation}
then 
\begin{equation}
	\nabla_{\mu}{\phi}<-1 \quad \text{along}\,\, \partial\Sigma_t.
\end{equation}
Hence, the maximum of $\phi$ is attained in the interior of $\Sigma_t$. On the one hand,
\begin{align*}
	\partial_t\langle x,e\rangle=&\frac{1}{F}\langle\nu,e\rangle+\langle \mathcal{T},e\rangle=\frac{1}{F}\langle\nu,e\rangle+\langle \mathcal{T},\nabla\langle x,e\rangle\rangle,\\
	\frac{1}{F^2}\dot{F}^{k\ell}\nabla_k\nabla_{\ell}\langle x,e\rangle=&\frac{1}{F^2}\dot{F}^{k\ell}(-h_{k\ell}\langle\nu,e\rangle)=-\frac{1}{F}\langle\nu,e\rangle.
\end{align*}
It follows that
\begin{equation}\label{Lheight}
	\mathcal{L}\langle x,e\rangle=\frac{2}{F}\langle\nu,e\rangle.
\end{equation}
On the other hand, taking trace of \eqref{evo-b}, we have
\begin{align}
	\mathcal{L}{\tilde{H}}=&-\frac{1}{F^2}(b^2)^{ij}\ddot{F}^{k\ell,pq}\nabla_i{h_{k\ell}}\nabla_j{h_{pq}}+\frac{2}{F^3}(b^2)^{ij}\nabla_i F\nabla_j F\notag\\
	&-\frac{2}{F^2}(b^2)^{ip}b^{qs}\dot{F}^{k\ell}\nabla_k{h_{pq}}\nabla_{\ell}{h_{si}}-\frac{1}{F^2}\dot{F}^{k\ell}(h^2)_{k\ell}\tilde{H}+\frac{2n}{F}\notag\\
	\leq& -\frac{1}{F^2}\dot{F}^{k\ell}(h^2)_{k\ell}\tilde{H}+\frac{2n}{F},\label{L-H}
\end{align}
where we used the Codazzi equation  and Lemma \ref{lem-ger96}. Combining \eqref{Lheight} with \eqref{L-H}, we obtain
\begin{align}
	\mathcal{L}\phi=&\frac{\mathcal{L}{\tilde{H}}}{\tilde{H}}-\frac{\beta}{\langle x,e\rangle}\mathcal{L}{\langle x,e\rangle}+\frac{\dot{F}^{k\ell}}{F^2\tilde{H}^2}\nabla_k{\tilde{H}}\nabla_{\ell}{\tilde{H}}-\frac{\beta}{F^2\langle x,e\rangle^2}\dot{F}^{k\ell}\langle e_k,e\rangle\langle e_{\ell},e\rangle-\alpha\notag\\
	\leq &-\frac{1}{F^2}\dot{F}^{k\ell}(h^2)_{k\ell}+\frac{2n}{F\tilde{H}}-\frac{2\beta}{\langle x,e\rangle F}\langle\nu,e\rangle+\frac{\dot{F}^{k\ell}}{F^2\tilde{H}^2}\nabla_k{\tilde{H}}\nabla_{\ell}{\tilde{H}}\notag\\
	&-\frac{\beta}{F^2\langle x,e\rangle^2}\dot{F}^{k\ell}\langle e_k,e\rangle\langle e_{\ell},e\rangle-\alpha.\label{L-phi-1}
\end{align}
At the maximum point of $\phi$ in $\Sigma_t$ for $t\in[0,\bar{T})$, we have $\nabla\phi=0$ and $\frac{1}{F^2}\dot{F}^{k\ell}\nabla_k\nabla_{\ell}\phi\leq 0$ , which means
\begin{equation}\label{Con-naphi}
	\nabla_i{\tilde{H}}=\frac{\beta\tilde{H}}{\langle x,e\rangle}\langle e_i,e\rangle.
\end{equation}
Plugging \eqref{Con-naphi} into \eqref{L-phi-1}, we derive
\begin{align}
	\partial_t\phi\leq\mathcal{L}{\phi}\leq&-\frac{1}{F^2}\dot{F}^{k\ell}(h^2)_{k\ell}+\frac{2n}{F\tilde{H}}-\frac{2\beta}{\langle x,e\rangle F}\langle\nu,e\rangle\notag\\
	&+\frac{\beta(\beta-1)}{F^2\langle x,e\rangle^2}\dot{F}^{k\ell}\langle e_k,e\rangle\langle e_{\ell},e\rangle-\alpha,
\end{align}
Note that both $F$ and $H$ are uniformly bounded from below and above. Therefore,  there exists a constant $c_5>0$, such that $F\geq c_5 H$, and hence
\begin{equation}
	F\tilde{H}\geq c_5 H\tilde{H}\geq n^2c_5.
\end{equation}
Similarly, all other terms are uniformly bounded. It follows that there exists a constant $C$ such that
\begin{equation}
\frac{d}{dt}(\max_{\Sigma_t}{\phi})\leq C-\alpha.
\end{equation}
If we choose $\alpha=C+1$, then we have $\frac{d}{dt}(\max_{\Sigma_t}{\phi})\leq0$. Therefore, we conclude that for all $t\in[0, \bar{T})$
\begin{align*} 
	\log\tilde{H}-\beta\log\langle x,e\rangle-\alpha t\leq \max_{\Sigma_0}(	\log\tilde{H}-\beta\log\langle x,e\rangle),
\end{align*}
which implies that $\tilde{H}$ has a uniform upper bounded on $\Sigma_{{t}}$ $(t\in[0,\bar{T}))$. This contradicts the definition of $\bar{T}$. Thus, we must have
 $\partial\Sigma_{\bar{T}}=\partial C_{\theta,\infty}(e)$ for some $e\in\mathbb{S}^n$, and then $\bar{T}=T^{*}$. 
\end{proof}
\begin{cor}
	There holds $T^{* }<\infty$.
\end{cor}
\begin{proof}
	Since the flow \eqref{Inverse-flow} induced a normal hypersurface flow $\partial\Sigma_t\subset\mathbb{S}^n$ with normal speed $\frac{1}{F\sin\theta}$, and since $F$ admits a uniform upper bound, it follows that the speed for $\partial\Sigma_t\subset\mathbb{S}^n$ has a uniform positive lower bound.  Consequently, the limiting hypersurface $\partial\Sigma_{\bar{T}}$ is reached at finite time.
\end{proof}
\section{Convergence to a flat ball with capillary boundary}\label{Sec-Con}
In this section, we conclude that there exists a constant vector $e\in\mathbb{S}^n$, such that the flow \eqref{Inverse-flow} converges to $C_{\theta,\infty}(e)$. For this purpose, we prove the following proposition.
\begin{prop}
		Assume that $\Sigma\subset\bar{\mathbb{B}}^{n+1} (n\geq 2)$ is a properly embedded, strictly convex smooth hypersurface with capillary boundary supported on $\mathbb{S}^n$ at a contact angle $\theta\in(0,\frac{\pi}{2}]$. Let $\Sigma_t,t\in[0,T)$ be the solution of the flow \eqref{Inverse-flow} starting from $\Sigma$. Assume that $\Sigma_{T}$ is not a flat ball around some point $e\in\mathbb{S}^n$, then there holds
		\begin{equation}
			T^{*}>T+\epsilon,
		\end{equation}
	where $\epsilon>0$ is a constant depending on $\Sigma$ and the distance between $\partial\Sigma_T$ and some suitable $\partial C_{\theta,\infty}(e)$.
\end{prop}
\begin{proof}
	Firstly we reduce the flow \eqref{Inverse-flow} to a scalar parabolic flow with oblique boundary condition using the conformal transformation map as in \cites{Scheuer-Wang-Xia2018,Weng-Xia2022}. Since $\partial\Sigma_T$ is not a flat ball, then by Lemma \ref{lem-rigidity}, there exits a vector $e\in \mathrm{int}(\widehat{\partial\Sigma_{T}})$, such that $\partial\Sigma_{T}\subset\mathrm{int}(\widehat{\partial C_{\theta,\infty}}(e))$. Without loss of generality, we assume that $e$ is the $(n+1)$th coordinate vector in $\mathbb{R}^{n+1}$. Consider the following M\"obius transformation:
	\begin{align}\label{conform-transform}
		\varphi: \bar{\mathbb B}^{n+1} \quad &\longrightarrow \quad \bar{\mathbb R}_{+}^{n+1},\notag\\
		(x',x_{n+1})  &\longmapsto  \frac{2(x',0)+(1-|x'|^2-x_{n+1}^2)e}{|x'|^2+(x_{n+1}-1)^2}:=(y',y_{n+1})=\tilde{y},
	\end{align}
	where $x=(x',x_{n+1})$ with $x'=(x_1,\cdots,x_n)\in \mathbb R^{n}$ and $x_{n+1}=\langle x,e\rangle$. Then $\varphi$ maps $\mathbb S^n=\partial \mathbb B^{n+1}$ to $\partial \mathbb R_{+}^{n+1}:=\{(y',y_{n+1})\in \mathbb R^{n+1}:y_{n+1}=0\}$ and
	\begin{align*}
		\varphi^\ast(\delta_{\mathbb R^{n+1}_{+}})=\frac{4}{(|x'|^2+(x_{n+1}-1)^2)^2}\delta_{\mathbb B^{n+1}}:=e^{-2w}\delta_{\mathbb B^{n+1}},
	\end{align*}
	which implies that $\varphi$ is a conformal transformation from $(\bar{\mathbb B}^{n+1},\delta_{\mathbb B^{n+1}})$ to $(\bar{\mathbb R}_{+}^{n+1},\delta_{\mathbb R_+^{n+1}})$. Then a properly embedded hypersurface $\Sigma_t=x(\bar{\mathbb B}^n,t)$ in $(\bar{\mathbb B}^{n+1},\delta_{\mathbb B^{n+1}})$ can be identified with $\widetilde{\Sigma}_t:=\tilde{y}(\bar{\mathbb B}^n,t)$ in $(\bar{\mathbb R}_{+}^{n+1},(\varphi^{-1})^\ast\delta_{\mathbb R_+^{n+1}})$, where $\tilde{y}=\varphi\circ x$.
	
	Note that $X_e$ is a conformal vector field such that $\varphi_{\ast}(X_e)=-\tilde{y}$. For a hypersurface $\Sigma \subset \bar{\mathbb B}^{n+1}$ with capillary boundary $\partial \Sigma\subset \mathbb S^n$, one has
	\begin{align*}
		e^{-2w}\langle X_e,\nu\rangle=\langle \varphi_\ast(X_e),\varphi_\ast(\nu)\rangle=|\varphi_{\ast}(\nu)|\langle \tilde y,\tilde\nu\rangle,
	\end{align*}
	where $|\varphi_\ast(\nu)|=\frac{1}2(|y'|^2+(y_{n+1}+1)^2)$ and $\tilde{\nu}:=-\frac{\varphi_\ast(\nu)}{|\varphi_\ast(\nu)|}$. Hence, the hypersurface $\varphi(\Sigma)$ is star-shaped in $\mathbb R_{+}^{n+1}$ with respect to the origin, i.e., $\langle \tilde{y},\tilde\nu\rangle>0$ on $\varphi(\Sigma)$,  if and only if $\langle X_e,\nu\rangle>0$  on $\Sigma$. In particular, since $\langle X_e,\nu\rangle>0$ on $\Sigma_t$ by \cite[Proposition 2.5]{HWYZ-2023}(see also \cite[Proposition 2.16]{Weng-Xia2022}), the hypersurface $\widetilde{\Sigma}_t$ in $(\-{\mathbb R}_{+}^{n+1},(\varphi^{-1})^\ast\d_{\mathbb R_+^{n+1}})$ can be written as a radial graph over $\-{\mathbb S}^{n}_{+}$.
	
	In $\mathbb R_{+}^{n+1}$, we use the polar coordinate $\tilde{y}=(\rho,z)\in [0,\infty)\times \bar{\mathbb{S}}_+^n$, where $\rho$ is the distance from $\tilde{y}$ to the origin and we write $z=(\b,\xi)\in [0,\frac{\pi}{2}]\times \mathbb S^{n-1}$ for the spherical polar coordinate of $z\in \mathbb S^{n}$. Then
	\begin{align}\label{dn-rho}
		\left\{\begin{aligned}
			\rho^2=&|y'|^2+y_{n+1}^2,\\
			y_{n+1}=&\rho \cos \b,\quad |y'|=\rho \sin \b.
		\end{aligned}   \right.
	\end{align}
	Let $u:=\log \rho$ and $v:=\sqrt{1+|\nabla^{\mathbb S}u|^2}$, where $\nabla^{\mathbb S}$ is the Levi-Civita connection on $\mathbb S^{n}_+$ with respect to the round metric $\s$. Then up to a time-dependent tangential diffeomorphism, one can rewrite the flow \eqref{Inverse-flow} equivalently as the following scalar parabolic equation on $\-{\mathbb S}_{+}^{n}$:
	\begin{equation}\label{Inverse-scalar-flow}
		\left\{\begin{aligned}
			\frac{\partial}{\partial t} u&=\frac{v}{\rho\mathrm{e}^{\omega}}\frac{1}{F(\nabla^2_{\mathbb S} u, \nabla^\mathbb S u, \rho, \b)} \qquad &\text{in}\quad \mathbb S_{+}^n \times[0,T),\\
			\nabla^{\mathbb S}_{\partial_\b}u&= -\cos\t \sqrt{1+|\nabla^{\mathbb S}u|^2}\qquad &\text{on}\quad \partial\mathbb S_+^{n} \times[0,T),\\
			u(\cdot,0)&=u_0(\cdot) \qquad &\text{on} \quad \mathbb S_{+}^n,
		\end{aligned}\right.
	\end{equation}
	where $u_0=\log\rho_{0}$, $\rho_0$ is radial function of the initial hypersurface $\varphi(\Sigma_0)$, $\partial_\b$ is the unit outward normal of $\partial \mathbb S_+^n$ on $\-{\mathbb S}_+^{n}$. 
	
	Since $\partial\Sigma_{T}\subset\mathrm{int}(\widehat{\partial C_{\theta,\infty}}(e))$, by the argument in Lemma \ref{Lem-pconvex}, we see that the maximum of $\phi$ is monotone decreasing along the flow \eqref{Inverse-flow}, which yields an uniform lower bound for principal curvatures of $\Sigma_t$, $t\in[0,T)$. This in turn implies that $u$ is uniformly bounded in $C^2(\mathbb{S}^n_{+}\times[0,T))$ and the scalar equation \eqref{Inverse-scalar-flow} is uniformly parabolic. Note that $|\cos\theta|<1$, the boundary value condition satisfies the uniform oblique property as in \cite{Ural1991}. Then by the parabolic theory in \cite[Theorem 5]{Ural1991} (see also \cite[\S 14]{Lieb96}), we conclude the uniform $C^\infty$-estimates and the solution can be extended to time $t=T$ smoothly, which implies the maximum existence time can be extended beyond time $T$ by a small positive constant $\epsilon$.
\end{proof}
As a corollary, we present a refined estimate of the height function which should be of independent interest in the future.
\begin{thm}
	Let $\Sigma\subset\bar{\mathbb{B}}^{n+1} (n\geq 2)$ be a properly embedded, strictly convex smooth hypersurface with capillary boundary supported on $\mathbb{S}^n$ at a contact angle $\theta\in(0,\frac{\pi}{2}]$, given by an embedding $x:\bar{\mathbb{B}}^n\to\bar{\mathbb{B}}^{n+1}$. Moreover, assume that the inverse curvature flow \eqref{Inverse-flow} with $F=\frac{E_{k}}{E_{k-1}}(1\leq k\leq n+1)$, starting from $\Sigma$ converges to the flat ball $C_{\theta,\infty}(e)$ around $e$.
	Then, the height function on $\Sigma$ satisfies
	\begin{align}\label{est-height}
		\langle x,e\rangle-\cos\theta\geq \Lambda\frac{\log(	W_{k,\theta}(\widehat{C}_{\theta,\infty}(e))+C_{n,k,\theta})-\log(	W_{k,\theta}(\widehat{\Sigma})+C_{n,k,\theta})}{(n+1-k)\max_{\Sigma}F},
 	\end{align}
 	where $\Lambda$ is a constant only depending on the distance of $e$ to $\partial\Sigma$, $C_{n,k,\theta}$ is a constant depending on $n,k$ and $\theta$.
\end{thm}
\begin{proof}
	From the initial hypersurface $\Sigma$, we start the inverse curvature flow \eqref{Inverse-flow} with $F=\frac{E_{k}}{E_{k-1}}$. Then,  according to \eqref{s2.dW}, $W_{k,\theta}$ evolves in time by
	\begin{align*}
	\frac{d}{dt}W_{k,\theta}(\widehat{\Sigma}_t)=\frac{n+1-k}{n+1}\int_{\Sigma_t}E_{k-1}d\mu_t\leq(n+1-k)(W_{k,\theta}(\widehat{\Sigma}_t)+C_{n,k,\theta}),
	\end{align*}
	where the inequality follows from the definition of $W_{k,\theta}$ and the fact $0\leq W_{\ell,\theta}(\widehat{\Sigma}_t)\leq W_{\ell,\theta}(\widehat{C}_{\theta,\infty})$, $C_{n,k,\theta}$ is a constant depending only on $n,k$ and $\theta$. Thus, we have the estimate
	\begin{align*}
		W_{k,\theta}(\widehat{\Sigma}_t)+C_{n,k,\theta}\leq (W_{k,\theta}(\widehat{\Sigma})+C_{n,k,\theta})e^{(n+1-k)t}.
	\end{align*}
	Since the flow \eqref{Inverse-flow} converges to the flat ball $\widehat{C}_{\theta,\infty}(e)$, the maximal existence time $T^*$ is at least
	\begin{align}\label{est-T-eq1}
		T^*\geq\frac{1}{n+1-k}\left(\log(	W_{k,\theta}(\widehat{C}_{\theta,\infty})+C_{n,k,\theta})-\log(	W_{k,\theta}(\widehat{\Sigma})+C_{n,k,\theta})\right)
	\end{align} 
	On the other hand, along the flow \eqref{Inverse-flow} with $F=\frac{E_{k}}{E_{k-1}}$,
	\begin{align*}
		\frac{d}{dt}\langle x,e\rangle=\frac{\langle\nu,e\rangle}{F}.
	\end{align*}
	Using the argument in \cite[Lemma 11]{Lambert-Scheuer2016}, we know that there exists a positive constant $\Lambda$ depending only on the distance of $e$ to $\partial\Sigma$ such that $\langle\nu,e\rangle\leq -\Lambda$. Therefore, by the flow \eqref{Inverse-flow} converges to the flat ball $\widehat{C}_{\theta,\infty}(e)$, we derive
	\begin{align}\label{est-T-eq2}
		\cos\theta-\langle x,e\rangle =\int_{0}^{T^*}\frac{d}{dt}\langle x,e\rangle dt=\int_{0}^{T^*}\frac{\langle\nu,e\rangle}{F}dt\leq-\frac{\Lambda}{\max_{\Sigma}F}T^*,
	\end{align}
	where the inequality follows from Lemma \ref{upper-F-H} and the distance of $e$ to $\partial\Sigma_t$ is increasing along the expanding flow \eqref{Inverse-flow}. Thus, combining \eqref{est-T-eq1} with \eqref{est-T-eq2} gives the estimate \eqref{est-height}.
\end{proof}

\section{Proof of Theorem \ref{Thm-A F}}\label{Sec-A F}
In this section, we prove a family of Alexandrov-Fenchel inequalities \eqref{In-A-F} for weakly convex hypersurfaces with free boundary in $\bar{\mathbb{B}}^{n+1}$ using the inverse mean curvature flow. We review the definition of quermassintegral for weakly convex hypersurfaces with free boundary here by taking $\theta=\frac{\pi}{2}$ in \S\ref{sub-sec2.2}.
\begin{align}
	W_0(\widehat{\Sigma})&=|\widehat{\Sigma}|,\quad W_1(\widehat{\Sigma})=\frac{1}{n+1}|\Sigma|,\label{Defn-W01}\\
	W_k(\widehat{\Sigma})&=\frac{1}{n+1}\int_{\Sigma}{E_{k-1}}\,dA+\frac{1}{n+1}\frac{k-1}{n-k+2}W_{k-2}^{\mathbb{S}^n}(\widehat{\partial\Sigma}),\quad \forall 2\leq k\leq n+1.\label{Defn-Wk}
\end{align}
Firstly, we need the following useful lemmas.
\begin{lem}\label{Lem-5.1}
	For any strictly convex hypersurface $\Sigma$ in $\bar{\mathbb{B}}^{n+1}$ with free boundary, there holds
	\begin{equation}\label{In-xnu}
		\langle x,\nu\rangle\leq 0.
	\end{equation}
\end{lem}
\begin{proof}
It follows from \cite[Lemma 11, Lemma 12]{Lambert-Scheuer2016} that there exists $e\in\mathrm{int}(\widehat{\partial\Sigma})$ such that $\partial\Sigma$ lies on the open upper hemisphere centered at $e$. Moreover, there exist constants $c_0>0$ depending on the distance of $e$ to $\partial\Sigma$ and $\delta_0>0$ depending on $c_0$, the length of second fundamental form of $\Sigma$ and the distance of $\partial\Sigma$ to the equator $\mathcal{S}(e)$ such that
\begin{equation}\label{Bounds-1}
	\langle\nu,e\rangle\leq -c_0,\quad \langle x,e\rangle\geq \delta_0.
\end{equation}
Note that $\langle x,\nu\rangle\equiv 0$ on $\partial\Sigma$. If $\langle x,\nu\rangle$ attains its positive maximum at some interior point $p$, then at  $p$ we have
\begin{equation}
	\langle x,\nu\rangle_i=h_i^k\langle x,e_k\rangle=0.
\end{equation}
Since $\Sigma$ is strictly convex, we have $\langle x,e_k\rangle=0$ for all $k$ and it follow that $x//\nu$ at the point $p$. By \eqref{Bounds-1}, we obtain $x=-|x|\nu$ at the point $p$, and therefore
\begin{equation}
	\langle x(p),\nu(p)\rangle=-|x(p)|\leq -\delta_0,
\end{equation}
which contradicts the assumption that $\langle x,\nu\rangle$ attains its positive maximum at $p$. Thus, for any point on $\Sigma$, we must have $\langle x,\nu\rangle\leq 0$.
\end{proof}

\begin{lem}\label{es-con}
	For any strictly convex hypersurface $\Sigma$ in $\bar{\mathbb{B}}^{n+1}$ with free boundary, there exists a constant $c_1>0$, depending on $n$ and the constant $\delta$ in Lemma \ref{Lem-es-height}, such that
	\begin{equation}\label{In-con1}
		(n+1)W_1(\widehat{\Sigma})-W_1^{\mathbb{S}^n}(\widehat{\partial\Sigma})\leq -c_1.
	\end{equation}
\end{lem}
\begin{proof}
	By \eqref{W_1} and \eqref{Defn-W01}, \eqref{In-con1} is equivalent to
	\begin{equation}\label{In-con2}
		|\Sigma|-\frac{1}{n}|\partial\Sigma|\leq -c_1.
	\end{equation}
Similarly as in the proof of Lemma \ref{Lem-es-height}, we define the cone $C_{\Sigma}$ as follows:
\begin{equation}\label{Defn-convex cone}
	C_{\Sigma}=\{y\in\mathbb{R}^{n+1}:y=sx, 0\leq s\leq 1,x\in\partial\Sigma\}. 
\end{equation}
Moreover, taking $e$ and $\delta$ as in Lemma \ref{Lem-es-height} and denote
\begin{equation}\label{Defn-tildecone}
	D_{\Sigma}=\left(\widehat{C_{\Sigma}}\cap \{x\in\bar{\mathbb{B}}^{n+1}:\langle x,e\rangle=\delta\}\right)\cup\left(C_{\Sigma}\cap \{x\in\bar{\mathbb{B}}^{n+1}:\langle x,e\rangle\geq\delta\}\right).
\end{equation}
It follows easily that all $\widehat{\Sigma}, \widehat{D_{\Sigma}}, \widehat{C_{\Sigma}}$ are geodesic convex in $\mathbb{R}^{n+1}$ and satisfy $\widehat{\Sigma}\subset \widehat{{D}_{\Sigma}}\subset\widehat{C_{\Sigma}}$. By a direct observation, for the cone $C_{\Sigma}$, there holds
\begin{equation}\label{eq-re-cone}
	\frac{|C_{\Sigma}|}{|\partial C_{\Sigma}|}=\frac{|C_{\Sigma}|}{|\partial\Sigma|}=\frac{|\mathbb{B}^n|}{|\mathbb{S}^{n-1}|}=\frac{1}{n}.
\end{equation}
On the other hand, according to \cite[III.13.2]{San04}, for a geodesic convex domain in $\mathbb{R}^{n+1}$, there holds
\begin{align}
	W_1^{\mathbb{R}^{n+1}}(\widehat{\Sigma})&=\frac{1}{n+1}(|\Sigma|+|\widehat{\partial\Sigma}|),\label{W_1Sigma}\\
	W_1^{\mathbb{R}^{n+1}}(\widehat{D_{\Sigma}})&=\frac{1}{n+1}(|D_{\Sigma}|+|\widehat{\partial\Sigma}|),\label{W_1tildeCSigma}\\
	W_1^{\mathbb{R}^{n+1}}(\widehat{C_{\Sigma}})&=\frac{1}{n+1}(|C_{\Sigma}|+|\widehat{\partial\Sigma}|).\label{W_1CSigma}
\end{align}
Then, combined these with \eqref{eq-re-cone}, we have
\begin{align}
	|\Sigma|-\frac{1}{n}|\partial\Sigma|=&|\Sigma|-|C_{\Sigma}|\notag\\
	=&(n+1)W_1^{\mathbb{R}^{n+1}}(\widehat{\Sigma})-(n+1)W_1^{\mathbb{R}^{n+1}}(\widehat{C_{\Sigma}})\notag\\
	\leq&(n+1)W_1^{\mathbb{R}^{n+1}}(\widehat{D_{\Sigma}})-(n+1)W_1^{\mathbb{R}^{n+1}}(\widehat{C_{\Sigma}})\notag\\
	=&|D_{\Sigma}|-|C_{\Sigma}|\notag\\
	=&-c_1<0,
\end{align}
where in the inequalities we used the monotonicity of quermassintegrals with respect to the set inclusion and the constant $c_1$ depends on the constant $\delta$. This completes the proof of Lemma \ref{es-con}.
\end{proof}
Similarly, for $k\geq 2$, we have the following weaker estimate:
\begin{lem}\label{Lem-con1}
		For any strictly convex hypersurface $\Sigma$ in $\bar{\mathbb{B}}^{n+1}$ with free boundary, and for all $2\leq k\leq n$, there holds
		\begin{equation}
			(n+1)W_k(\widehat{\Sigma})\leq W_k^{\mathbb{S}^n}(\widehat{\partial\Sigma}).
		\end{equation}
\end{lem}
\begin{proof}
Let $\{e_i\}_{i=1}^n$ be a local orthonormal basis of $T\Sigma$. Note that for $2\leq k\leq n$, there holds
\begin{align}
	&\mathrm{div}_{\Sigma}(\dot{E}_k^{ij}\langle x,e_i\rangle e_j)\notag\\
	=&\dot{E}_k^{ij}(g_{ij}-\langle x,\nu\rangle h_{ij})\notag\\
	=&kE_{k-1}-k\langle x,\nu\rangle E_k\notag\\
	\geq&kE_{k-1},\label{DivE_k}
\end{align}
where in the last inequality we  used Lemma \ref{Lem-5.1}. Hence we obtain
\begin{align}
	&\int_{\Sigma}{\mathrm{div}_{\Sigma}(\dot{E}_k^{ij}\langle x,e_i\rangle e_j)}\,dA\notag\\
	=&\int_{\partial\Sigma}{\dot{E}_k^{ij}\langle x,e_i\rangle\langle\mu,e_j\rangle}\,ds\notag\\
	=&\int_{\partial\Sigma}{\dot{E}_k^{\mu\mu}}\,ds=\binom{n}{k}^{-1}\int_{\partial\Sigma}{\sigma_{k-1}(\kappa|h_{\mu\mu})}\,ds\notag\\	
	=&\frac{k}{n}\int_{\partial\Sigma}{E_{k-1}^{\mathbb{S}^n}}\,ds\notag\\
	\geq&k\int_{\Sigma}{E_{k-1}}\,dA,
\end{align}
which implies
\begin{equation}\label{Ineq-Int}
	\int_{\partial\Sigma}{E_{k-1}^{\mathbb{S}^n}}\,ds\geq n\int_{\Sigma}{E_{k-1}}\,dA.
\end{equation}
Combining \eqref{W_k}, \eqref{Defn-Wk} with \eqref{Ineq-Int} completes the proof of Lemma \ref{Lem-con1}.
\end{proof}

Lambert and Scheuer \cite[Lemma 3.4]{Lambert-Scheuer2017} proved that for a weakly convex hypersurface $\Sigma$ with free boundary in $\bar{\mathbb{B}}^{n+1}$, $|\Sigma|$ has a positive upper bound depending on $\partial\Sigma$. Precisely, they proved the following estimate.

\begin{lem}[\cite{Lambert-Scheuer2017}]\label{lem-area-L}
	Let $\Sigma$ be a weakly convex hypersurface with free boundary in $\bar{\mathbb{B}}^{n+1}$ such that $\partial\Sigma$ is not an equator. Then there holds
	\begin{equation}
		|\Sigma|\leq b_n-c_{\partial\Sigma},
	\end{equation}
	where $b_n=\vert\mathbb{B}^n\vert$ is the volume of the  n-dimensional unit disk $\mathbb{B}^n$,  $c_{\partial\Sigma}$ is a constant only depending on the outer radius of $\partial\Sigma\subset\mathbb{S}^n$, in the sense that it tends to zero only if the outer radius tends to $\frac{\pi}{2}$.
\end{lem}

Then in the final part of this subsection, we prove that for a weakly convex hypersurface $\Sigma$ with free boundary in $\bar{\mathbb{B}}^{n+1}$, $|\Sigma|$ has a positive lower bound depending on $\partial\Sigma$. Precisely, we have
\begin{lem}\label{Lower-bound}
	Let $\Sigma$ be a weakly convex hypersurface with free boundary in $\bar{\mathbb{B}}^{n+1}$ such that $\partial\Sigma$ is not an equator. Then there holds
	\begin{equation}
		|\Sigma|\geq d_{\partial\Sigma},
	\end{equation}
	where $d_{\partial\Sigma}$ is a constant depending on the inner radius of $\partial\Sigma\subset\mathbb{S}^n$, in the sense that it tends to zero only if the inner radius tends to 0.
\end{lem}
\begin{proof}
	Without loss of generality, we assume that $\Sigma$ is strictly convex, the weakly convex case then follows by approximation. For any $e\in\mathrm{int}(\widehat{\partial\Sigma})$ such that $\widehat{\partial\Sigma}$ is contained in the open hemisphere $\mathcal{H}(e)$, we claim that the maximum of $\langle x,e\rangle$ is attained on $\partial\Sigma$. Suppose, for contradiction, that the strict maximum is attained at some interior point
	$p\in\mathrm{int}(\Sigma)$. Then we have $\langle x,e\rangle_i=0$,  and hence the unit normal satisfies $\nu(p)//e$ at $p$. By Lemma \ref{Lem-es-height} and Lemma \ref{Lem-5.1}, we deduce that $\nu(p)=-e$. On the other hand, since $\Sigma$ is strictly convex, it lies on one side of the tangent plane $T_p\Sigma$, on the side opposite to $\nu(p)$. More precisely, we have
	\begin{equation*}
		\langle x-p,\nu(p)\rangle\leq 0,\forall x\in\Sigma,
	\end{equation*}
which implies
\begin{equation*}
\langle x,e\rangle\geq \langle p,e\rangle, \forall x\in\Sigma,
\end{equation*}
contradicting the assumption that $\langle x,e\rangle$ achieves a strict maximum at the interior point  $p$. Therefore,  the maximum of $\langle x,e\rangle$ must be attained on $\partial\Sigma$. It follows that for a fixed $e$, there exists a $\delta'>0$ depending $\partial\Sigma$, such that 
\begin{equation}
	\{x\in\bar{\mathbb{B}}^{n+1}:\langle x,e\rangle\geq\delta'\}\subset \widehat{\Sigma}.
\end{equation}
This implies the lower bound
\begin{equation*}
	|\Sigma|\geq |\{x\in\bar{\mathbb{B}}^{n+1}:\langle x,e\rangle=\delta'\}|=b_n\left[1-(\delta')^2\right]^{\frac{n}{2}}:=d_{\partial\Sigma}.
\end{equation*}
\end{proof}

\subsection{The case of strictly convex hypersurfaces} In this subsection, we show that  the strict inequality  in \eqref{In-A-F} holds for all strictly convex hypersurfaces $\Sigma$. 
\begin{prop}\label{Prop-lim}
	Let $\Sigma$ be a strictly convex hypersurface in $\bar{\mathbb{B}}^{n+1}$ with free boundary and $\{\Sigma_t,0<t\leq T^*\}$ be the solution to the inverse mean curvature flow with initial hypersurface $\Sigma$, then for $k\in\mathbb{N}$ and $2k+1\leq n$, there holds
	\begin{equation}\label{Eq-W2k+1}
	 \lim_{t\to T^*}W_{2k+1}(\widehat{\Sigma}_t)=\frac{\omega_{n-1}}{n}\frac{(2k)!!(n-2k-1)!!}{(n+1)!!}.
	\end{equation}
\end{prop}
\begin{proof}
	By \cite[Lemma 2.2]{Lambert-Scheuer2017}, we have
	\begin{equation}\label{eq-InteE_1}
		\lim_{t\to T^{*}}\int_{\Sigma_t}{E_1^p}\,dA_t=0,\quad 1\leq p<\infty.
	\end{equation}
Since $\Sigma_{T^*}=C_{\frac{\pi}{2},\infty}(e)$ for some $e\in\mathbb{S}^n$, combining \eqref{Defn-Wk} and \eqref{eq-InteE_1} with  Lemma \ref{Lemma-NMI}, we obtain
\begin{equation}\label{eq-W_2k+1}
	\lim_{t\to T^{*}}W_{2k+1}(\widehat{\Sigma}_t)=\frac{2k}{(n+1)(n-2k+1)}W_{2k-1}^{\mathbb{S}^n}(B_{\frac{\pi}{2}}(e)),
\end{equation}
for $k\geq 1$, where $B_{\rho}(e)$ denotes the geodesic ball in $\mathbb{S}^n$ of radius $\rho$ centered at $e\in\mathbb{S}^n$. By \cite[Corollary 8]{Sol06}, for any domain $\Omega$ in $\mathbb{S}^n$ with $C^2$ boundary, the following identity holds:
\begin{equation*}
	W_{2k-1}^{\mathbb{S}^n}(\Omega)=\frac{1}{n}\sum_{i=0}^{k-1}\frac{(2k-2)!!(n-2k+1)!!}{(2k-2i-2)!!(n-2k+2i+1)!!}\int_{\partial\Omega}{E_{2k-2i-2}}\,ds.
\end{equation*}
In particular, for a geodesic ball $B_{\rho}\subset\mathbb{S}^n$,  this yields
\begin{equation*}
	W_{2k-1}^{\mathbb{S}^n}(B_{\rho})=\frac{\omega_{n-1}}{n}\sum_{i=0}^{k-1}\frac{(2k-2)!!(n-2k+1)!!}{(2k-2i-2)!!(n-2k+2i+1)!!}\cos^{2k-2i-2}\rho\sin^{n-2k+2i+1}\rho,
\end{equation*}
and hence
\begin{equation}\label{eq-Wequator}
	W_{2k-1}^{\mathbb{S}^n}(B_{\frac{\pi}{2}}(e))=\frac{\omega_{n-1}}{n}\frac{(2k-2)!!(n-2k+1)!!}{(n-1)!!}.
\end{equation}
Combining \eqref{eq-W_2k+1} and \eqref{eq-Wequator} yields \eqref{Eq-W2k+1} for $k\geq 1$. The case $k=0$ follows directly from \eqref{Defn-W01} and the fact that $b_n=\frac{\omega_{n-1}}{n}$.
\end{proof}
Next, we proceed by induction to prove that the strict inequality in \eqref{In-A-F} holds for all strictly convex hypersurfaces $\Sigma$. 
\begin{lem}\label{Lem-Strictly}
	Let $\Sigma\subset\bar{\mathbb{B}}^{n+1} (n\geq 2)$ be a properly embedded, strictly convex smooth hypersurface in the unit ball with free boundary. Then for $k\in\mathbb{N}_{+}$ and $2k+1\leq n$, there holds
	\begin{equation}\label{In-A-F-S}
		W_{2k+1}(\widehat{\Sigma})>\frac{\omega_{n-1}}{n}\frac{\prod_{j=0}^k(n-2j)}{\prod_{j=0}^k(n+1-2j)}\sum_{i=0}^k(-1)^i\binom{k}{i}\frac{1}{n-2k+2i}\left[\frac{n(n+1)}{\omega_{n-1}}W_1(\widehat{\Sigma})\right]^{\frac{n-2k+2i}{n}}.
	\end{equation}
\end{lem}
\begin{proof}
	We firstly prove \eqref{In-A-F-S} for $k=1$. Along the flow \eqref{Inverse-flow} with $F=\frac{1}{E_1}$, we have
	\begin{align}
		\partial_t W_1(\widehat{\Sigma}_t)&=\frac{n}{n+1}|\Sigma_t|=n W_1(\widehat{\Sigma}_t),\label{va-W1}\\
		\partial_t W_3(\widehat{\Sigma}_t)&=\frac{n-2}{n+1}\int_{\Sigma_t}{\frac{E_3}{E_1}}\,dA_t\notag\leq \frac{n-2}{n+1}\int_{\Sigma_t}{E_2}\,dA_t\notag\\
		&=(n-2)\left[W_3(\widehat{\Sigma}_t)-\frac{2}{(n+1)(n-1)}W_1^{\mathbb{S}^n}(\widehat{\partial\Sigma}_t)\right]\notag\\
		&<(n-2)W_3(\widehat{\Sigma}_t)-\frac{2(n-2)}{n-1}W_1(\widehat{\Sigma}_t),\label{va-W3}
	\end{align}
where we used Lemma \ref{es-con} in the last inequality. Combining \eqref{va-W1} and \eqref{va-W3} yields
\begin{equation*}
	\partial_t\left(W_3(\widehat{\Sigma}_t)+\frac{n-2}{n-1}W_1(\widehat{\Sigma}_t)\right)<(n-2)\left(W_3(\widehat{\Sigma}_t)+\frac{n-2}{n-1}W_1(\widehat{\Sigma}_t)\right).
\end{equation*}
If we define
\begin{equation}\label{Defn-Qt}
	Q(t)=W_1(\widehat{\Sigma}_t)^{-\frac{n-2}{n}}\left(W_3(\widehat{\Sigma}_t)+\frac{n-2}{n-1}W_1(\widehat{\Sigma}_t)\right),
\end{equation}
then,  by combining\eqref{va-W1} with \eqref{Defn-Qt}, we conclude that along the flow \eqref{Inverse-flow}, the following monotonicity property holds:
\begin{equation}\label{Mono-Q}
	\frac{d}{dt}Q(t)<0.
\end{equation}
By a direct calculation using Proposition \ref{Prop-lim}, we obtain
\begin{align*}
	\lim_{t\to T^{*}}W_{1}(\widehat{\Sigma}_t)&=\frac{\omega_{n-1}}{n(n+1)},\\
	\lim_{t\to T^{*}}W_{3}(\widehat{\Sigma}_t)&=\frac{2\omega_{n-1}}{n(n+1)(n-1)},
\end{align*}
and thus
\begin{equation}\label{Lim-Q}
	\lim_{t\to T^{*}} Q(t)=\frac{n^{\frac{n-2}{n}}}{n-1}\left(\frac{\omega_{n-1}}{n+1}\right)^{\frac{2}{n}},
\end{equation}
which establishes \eqref{In-A-F-S} for $k=1$.

 In the following, we assume that the strict inequality between $W_{2k-1}(\widehat{\Sigma}_t)$ and $W_1(\widehat{\Sigma}_t)$ holds. We then proceed by induction to show that the strict inequality between $W_{2k+1}(\widehat{\Sigma}_t)$ and $W_1(\widehat{\Sigma}_t)$ also holds. We denote 
 \begin{equation}\label{Defn-Ak}
 	A_k(\widehat{\Sigma}_t)=\frac{\omega_{n-1}}{n}\frac{\prod_{j=0}^k(n-2j)}{\prod_{j=0}^k(n+1-2j)}\sum_{i=0}^k(-1)^i\binom{k}{i}\frac{1}{n-2k+2i}\left[\frac{n(n+1)}{\omega_{n-1}}W_1(\widehat{\Sigma}_t)\right]^{\frac{n-2k+2i}{n}}
 \end{equation}
  and define the auxiliary function 
 \begin{equation}\label{Defn-phi}
 	\phi(t)=\frac{W_{2k+1}(\widehat{\Sigma}_t)-A_k(\widehat{\Sigma}_t)}{W_1(\widehat{\Sigma}_t)^{\frac{n-2k}{n}}},
 \end{equation}
 then by a direct calculation, we obtain
 \begin{equation}\label{de-phit}
 	\frac{d}{dt}\phi(t)=\frac{\frac{d}{dt}W_{2k+1}(\widehat{\Sigma}_t)-(n-2k)W_{2k+1}(\widehat{\Sigma}_t)}{W_1(\widehat{\Sigma}_t)^{\frac{n-2k}{n}}}-\frac{d}{dt}\left(\frac{A_k(\widehat{\Sigma}_t)}{W_1(\widehat{\Sigma}_t)^{\frac{n-2k}{n}}}\right).
 \end{equation}
Since
\begin{align}
	\frac{d}{dt}W_{2k+1}(\widehat{\Sigma}_t)=&\frac{n-2k}{n+1}\int_{\Sigma_t}{\frac{E_{2k+1}}{E_1}}\,dA_t
	\leq\frac{n-2k}{n+1}\int_{\Sigma_t}{E_{2k}}\,dA_t\notag\\
	=&(n-2k)\left[W_{2k+1}(\widehat{\Sigma}_t)-\frac{2k}{(n+1)(n-2k+1)}W_{2k-1}^{\mathbb{S}^n}(\widehat{\partial\Sigma}_t)\right]\notag\\
	\leq&(n-2k)W_{2k+1}(\widehat{\Sigma}_t)-\frac{2k(n-2k)}{n-2k+1}W_{2k-1}(\widehat{\Sigma}_t), \label{va-phit-1}
\end{align}
by combining \eqref{de-phit} with \eqref{va-phit-1}, we have
\begin{equation}\label{va-phit-2}
	\frac{d}{dt}\phi(t)\leq -\frac{2k(n-2k)}{n-2k+1}\frac{W_{2k-1}(\widehat{\Sigma}_t)}{W_1(\widehat{\Sigma}_t)^{\frac{n-2k}{n}}}-\frac{d}{dt}\left(\frac{A_k(\widehat{\Sigma}_t)}{W_1(\widehat{\Sigma}_t)^{\frac{n-2k}{n}}}\right).
\end{equation}
Now, using the assumption that the strict inequality between $W_{2k-1}(\widehat{\Sigma}_t)$ and $W_1(\widehat{\Sigma}_t)$ holds,  a direct computation yields
\begin{equation*}
	-\frac{2k(n-2k)}{n-2k+1}\frac{W_{2k-1}(\widehat{\Sigma}_t)}{W_1(\widehat{\Sigma}_t)^{\frac{n-2k}{n}}}-\frac{d}{dt}\left(\frac{A_k(\widehat{\Sigma}_t)}{W_1(\widehat{\Sigma}_t)^{\frac{n-2k}{n}}}\right)<0.
\end{equation*}
Therefore,
\begin{equation}
	\phi(0)>\lim_{t\to T^{*}}\phi(t).
\end{equation}
Next, we prove that $\lim_{t\to T^{*}}W_{2k+1}(\widehat{\Sigma}_{t})=\lim_{t\to T^*}A_k(\widehat{\Sigma}_t)$. By Proposition \ref{Prop-lim}, it is equivalent to prove
\begin{equation}\label{eq-iden}
	\frac{(2k)!!}{\prod_{j=0}^k(n-2j)}=\sum_{i=0}^k(-1)^i\binom{k}{i}\frac{1}{n-2k+2i}.
\end{equation}
If we denote the right hand side of \eqref{eq-iden} as $S(n,k)$, then we obtain
\begin{align}
	S(n,k)=&\sum_{i=0}^k (-1)^i\binom{k}{i}\frac{1}{n-2k+2i}\notag\\
	=&\sum_{i=0}^k (-1)^i\left[\binom{k-1}{i}+\binom{k-1}{i-1}\right]\frac{1}{n-2k+2i}\notag\\
	=&\sum_{i=0}^{k-1}(-1)^i\binom{k-1}{i}\frac{1}{n-2-2(k-1)+2i}-\sum_{i=1}^k(-1)^{i-1}\binom{k-1}{i-1}\frac{1}{n-2k+2i}\notag\\
	=&\sum_{i=0}^{k-1}(-1)^i\binom{k-1}{i}\frac{1}{n-2-2(k-1)+2i}-\sum_{i=0}^{k-1}(-1)^i\binom{k-1}{i}\frac{1}{n-2(k-1)+2i}\notag\\
	=&S(n-2,k-1)-S(n,k-1).\label{In-Snk}
\end{align}
It's easy to verify that \eqref{eq-iden} holds for $n=1$ and 2. Moreover, \eqref{eq-iden} also holds for the positive integer pairs $(n,k)=(n,0)$ and $(n,1)$ for all $n$. Suppose, for contradiction, that there exists a pair $(n,k)$ with $n\geq 3$ and $2k+1\leq n$ such that \eqref{eq-iden} does not hold. Let $n_0$ be the minimal integer for which \eqref{eq-iden} fails, and for this fixed $n_0$, let $k_0$ be the minimal integer such that \eqref{eq-iden} fails for $(n_0,k_0)$. Then by previous observation, we must have $n_0\geq 3$ and $k_0\geq 2$. However, we observe that
\begin{align}
	S(n_0,k_0)&=S(n_0-2,k_0-1)-S(n_0,k_0-1)\notag\\
	&=\frac{(2k_0-2)!!}{\prod_{j=0}^{k_0-1}(n_0-2-2j)}-\frac{(2k_0-2)!!}{\prod_{j=0}^{k_0-1}(n_0-2j)}\notag\\
	&=\frac{(2k_0)!!}{\prod_{j=0}^{k_0}(n_0-2j)},
\end{align}
which leads to a contradiction. This completes the proof of Lemma \ref{Lem-Strictly}.
\end{proof}

\subsection{The case of weakly convex hypersurfaces} In this subsection, we assume that $\Sigma$ is a weakly convex hypersurface. If $\Sigma$ is a flat disk with free boundary, then the equality in \eqref{In-A-F} clearly holds. If $\Sigma$ is not flat, then the inequality \eqref{In-A-F} follows from approximation using the mean curvature flow with free boundary.Therefore, in the remainder of this subsection, we prove that if $\Sigma$ is weakly convex but not flat, then the strict inequality in \eqref{In-A-F}  holds. The mean curvature flow for hypersurfaces in the unit ball with free boundary reads as:
\begin{equation}\label{flow-mean}
	\left\{\begin{aligned}
		\partial_t x&=-H\nu \qquad \text{in}\quad \bar{\mathbb{B}}^n \times[0,T),\\
		\langle\bar{N}\circ x,\nu\rangle&=0 \qquad \text{on}\quad \partial\bar{\mathbb{B}}^n \times[0,T),\\
		x(\cdot,0)&=x_0(\cdot) \qquad \text{on} \quad \bar{\mathbb{B}}^n.
	\end{aligned}\right.
\end{equation}
By \cite[Theorem 2.1]{Stahl1996-2}, there exists a solution of \eqref{flow-mean} for $x(\cdot,t)\in C^{\infty}(\bar{\mathbb{B}}^n\times(0,\varepsilon])\cap C^{2+\alpha,1+\frac{\alpha}{2}}(\bar{\mathbb{B}}^n\times[0,\varepsilon])$ for small $\varepsilon>0$ and we have the following approximation result.

\begin{thm}[\cite{Lambert-Scheuer2017}]
	Suppose $x(\cdot,t):\bar{\mathbb{B}}^n\times [0,T)\to\mathbb{R}^{n+1}$ is a solution to flow \eqref{flow-mean} with initial hypersurface $\Sigma_0$ being weakly convex with free boundary. Then either $\partial\Sigma_0$ is an equator of the unit sphere or $(h_{ij})>0$ for $t>0$.
\end{thm}

Firstly, we prove a stronger version of Lemma \ref{Lem-Strictly}, which will be needed in later analysis.

\begin{lem}\label{In-pre}
	Let $\Sigma\subset\bar{\mathbb{B}}^{n+1} (n\geq 2)$ be a properly embedded, strictly convex smooth hypersurface in the unit ball with free boundary. Then for $k\in\mathbb{N}_{+}$ and $2k+1\leq n$, there exists a constant $a_k>0$, depending only on $n,k,\partial\Sigma$ and the length of the second fundamental form of $\Sigma$ such that 
	\begin{equation}\label{In-A-F-S1}
		W_{2k+1}(\widehat{\Sigma})\geq\frac{\omega_{n-1}}{n}\frac{\prod_{j=0}^k(n-2j)}{\prod_{j=0}^k(n+1-2j)}\sum_{i=0}^k(-1)^i\binom{k}{i}\frac{1}{n-2k+2i}\left[\frac{n(n+1)}{\omega_{n-1}}W_1(\widehat{\Sigma})\right]^{\frac{n-2k+2i}{n}}+a_k.
	\end{equation}
\end{lem}
\begin{proof}
	We firstly prove \eqref{In-A-F-S1} for $k=1$. By Lemma \ref{Lem-es-height}, there exists a vector $e\in\mathrm{int}(\widehat{\partial\Sigma})$, a constant $\delta$ depending on $\partial\Sigma$ and the length of the second fundamental form of $\Sigma$, such that $\langle x,e\rangle\geq\delta$ for any $x\in\Sigma$. Then there exists a radius $0<R<\infty$, depending on $\delta$, such that 
	\begin{equation}
		\{x\in\bar{\mathbb{B}}^{n+1}:\langle x,e\rangle\geq \delta\}\subset \widehat{C_{\frac{\pi}{2},R}}(e).
	\end{equation}
	By a direct calculation, we conclude that the maximal existence times $T_1^{*}$ and $T_2^{*}$ for the flow \eqref{Inverse-flow} with $F=\frac{1}{E_1}$ starting from $\Sigma$ and $C_{\frac{\pi}{2},R}(e)$ satisfy
	\begin{align*}
		T_1^{*}&=\frac{1}{n}\log\left(\frac{b_n}{|\Sigma|}\right),\\
		T_2^{*}&=\frac{1}{n}\log\left(\frac{b_n}{|C_{\frac{\pi}{2},R}(e)|}\right).
	\end{align*}
	where $T_2^*$ depends only on $n$ and $R$, and we have $T_1^*>T_2^*$. We take 
	\begin{equation}
		\bar{T}=\frac{1}{2}T_2^{*}=\frac{1}{2n}\log\left(\frac{b_n}{|C_{\frac{\pi}{2},R}(e)|}\right).
	\end{equation}
	Assume that $\{\Sigma_t\}$ and $\{\widetilde{\Sigma}_t\}$ are the flow hypersurfaces to the flow \eqref{Inverse-flow} starting from $\Sigma$ and $C_{\frac{\pi}{2},R}(e)$ respectively. Then by the avoidance principle, we have $\widehat{\Sigma}_t\subset\widehat{\widetilde{\Sigma}}_t$ for $0\leq t\leq \bar{T}$. If we take $\delta_1>0$ as
	\begin{equation}\label{Choice-delta1}
		\delta_1=\min\{\langle x,e\rangle:x\in\widetilde{\Sigma}_{\bar{T}}\},
	\end{equation}
	we see easily that $\delta_1$ depends on $n$, $\partial\Sigma$ and the length of the second fundamental form of $\Sigma$. Moreover, for all $0\leq t\leq \bar{T}$, there holds
	\begin{equation}
		\widehat{\partial\Sigma}\subset\widehat{\partial\Sigma}_t\subset\widehat{\partial\tilde{\Sigma}}_{\bar{T}},
	\end{equation}
	where $\widehat{\partial\tilde{\Sigma}}_{\bar{T}}$ depends on $n$ and $R$. Hence by Lemma \ref{es-con}, there exists an uniform constant $c_1>0$, depending only on $n$ and $\delta_1$ such that 
	\begin{equation}\label{lower-W11}
		W_1^{\mathbb{S}^n}(\widehat{\partial\Sigma}_t)\geq (n+1)W_1(\widehat{\Sigma}_t)+c_1,\quad 0\leq t \leq \bar{T}.
	\end{equation}
	Then repeating the proof in Lemma \ref{Lem-Strictly} and using \eqref{lower-W11}, \eqref{Mono-Q} now becomes
	\begin{equation}\label{M-Q}
		\frac{d}{dt}Q(t)\leq \left\{\begin{aligned}
			&-\frac{2(n-2)c_1}{(n+1)(n-1)}W_1(\widehat{\Sigma}_t)^{-\frac{n-2}{n}},\,\,0<t\leq \bar{T},\\
			&0,\,\, \bar{T}<t\leq T_1^{*}.
		\end{aligned}\right.
	\end{equation}
	Applying Lemma \ref{lem-area-L}, we have
	\begin{equation*}
		W_1(\widehat{\Sigma}_t)=\frac{1}{n+1}|\Sigma_t|\leq \frac{b_n}{n+1},
	\end{equation*}
	then  for $0\leq t \leq \bar{T}$, there holds
	\begin{equation}\label{In-Q}
		\frac{d}{dt}Q(t)\leq -\frac{2(n-2)c_1}{(n+1)(n-1)}\left(\frac{b_n}{n+1}\right)^{-\frac{n-2}{n}},
	\end{equation}
	which in turn implies 
	\begin{align}
		W_3(\widehat{\Sigma})\geq& \frac{\omega_{n-1}}{(n+1)(n-1)}\left[\frac{n(n+1)}{\omega_{n-1}}W_1(\widehat{\Sigma})\right]^{\frac{n-2}{n}}-\frac{n-2}{n-1}W_1(\widehat{\Sigma})\notag\\
		&+\frac{2(n-2)c_1}{(n+1)(n-1)}\left(\frac{b_n}{n+1}\right)^{-\frac{n-2}{n}}W_1(\widehat{\Sigma})^{\frac{n-2}{n}}\bar{T}\notag\\
		\geq&\frac{\omega_{n-1}}{(n+1)(n-1)}\left[\frac{n(n+1)}{\omega_{n-1}}W_1(\widehat{\Sigma})\right]^{\frac{n-2}{n}}-\frac{n-2}{n-1}W_1(\widehat{\Sigma})\notag\\
		&+\frac{(n-2)c_1}{n(n+1)(n-1)}\left(\frac{d_{\partial\Sigma}}{b_n}\right)^{\frac{n-2}{n}}\log\left(\frac{b_n}{|C_{\frac{\pi}{2},R}(e)|}\right)\notag\\
		:=&\frac{\omega_{n-1}}{(n+1)(n-1)}\left[\frac{n(n+1)}{\omega_{n-1}}W_1(\widehat{\Sigma})\right]^{\frac{n-2}{n}}-\frac{n-2}{n-1}W_1(\widehat{\Sigma})+a_1,
	\end{align}
	where  we used Lemma \ref{Lower-bound} in the second inequality, and  $a_1>0$ depends on $n,\partial\Sigma$ and the length of the second fundamental form of $\Sigma$. 
	
Next, we assume that there holds
	\begin{equation*}
		W_{2k-1}(\widehat{\Sigma})\geq A_{k-1}(\widehat{\Sigma})+a_{k-1}.
	\end{equation*}
	The auxiliary function $\phi(t)$ is defined as in \eqref{Defn-phi}. Using the above assumption, by a direct calculation, we obtain
	\begin{equation}\label{Mono-phi1}
		\frac{d}{dt}\phi(t)\leq -\frac{2k(n-2k)}{n-2k+1}\frac{a_{k-1}}{W_1(\widehat{\Sigma}_t)^{\frac{n-2k}{n}}}\leq -\frac{2k(n-2k)}{n-2k+1}\left(\frac{b_n}{n+1}\right)^{-\frac{n-2k}{n}}a_{k-1},
	\end{equation}
	which implies
	\begin{align}
		W_{2k+1}(\widehat{\Sigma})&\geq A_k(\widehat{\Sigma})+\frac{2k(n-2k)}{n-2k+1}\left(\frac{b_n}{n+1}\right)^{-\frac{n-2k}{n}}a_{k-1}W_1(\widehat{\Sigma})^{\frac{n-2k}{n}}\notag\\
		&\geq A_k(\widehat{\Sigma})+\frac{2k(n-2k)}{n-2k+1}\left(\frac{d_{\partial\Sigma}}{b_n}\right)^{\frac{n-2k}{n}}a_{k-1}\notag\\
		&:=A_k(\widehat{\Sigma})+a_k.
	\end{align}
	This completes the proof of Lemma \ref{In-pre}.
\end{proof}
Assume that $\Sigma$ is weakly convex but not an equator. By Lemma \ref{Lem-es-height}, let $e\in\mathrm{int}(\widehat{\partial\Sigma})$ be an direction such that $\widehat{\partial\Sigma}$ is contained in the open hemisphere $\mathcal{H}(e)$. Then there exists $0<R_1<\infty$, such that $\widehat{C_{\frac{\pi}{2},R_1}}(e)\subset\mathrm{int}(\widehat{\Sigma})$. Let $\{\Sigma_t\}_{t\geq 0}$ be the mean curvature flow starting from $\Sigma$. Then there exists a $\varepsilon_0>0$, such that for all $0<t<\varepsilon_0$, there holds
\begin{align}
	\widehat{C_{\frac{\pi}{2},R_1}}(e)&\subset\widehat{\Sigma}_t\subset\mathrm{int}(\widehat{\Sigma}),\label{Re-subset}\\
	\max_{\Sigma_t}|A|^2&\leq \max_{\Sigma_0}|A|^2+1.\label{Re-A2}
\end{align}
Then by Lemma \ref{In-pre}, there exists a uniform constant $a_k$ depending on $n,k,\partial\Sigma$ and the length of the second fundamental form of $\Sigma$, but independent of time $t\in(0,\epsilon_0)$, such that \eqref{In-A-F-S1} holds.  Letting $t\to 0$, this yields the strict inequality in \eqref{In-A-F}    and thus completes the proof of Theorem \ref{Thm-A F}.

\end{document}